\documentclass[11pt]{amsart}
\usepackage[a4paper, margin=1in]{geometry}
\usepackage{amsfonts}
\usepackage{amssymb}
\usepackage{amsmath}
\usepackage{amsthm}
\usepackage{mathtools}
\usepackage{hyperref}
\usepackage{cleveref}
\usepackage{stmaryrd}
\usepackage{tikz-cd}
\usepackage{enumitem}
\usepackage{multirow}
\usepackage{bbm}
\usepackage{comment}
\usepackage{float}
\usepackage{url}

\title{On Converse Theorems for Hilbert Modular Forms Assuming Unramified Twists}
\author{Pengcheng Zhang}
\dedicatory{Dedicated to the memory of Ioan James (1928--2025)}
\address{Max Planck Institute for Mathematics, Vivatsgasse 7, 53111 Bonn, Germany}
\email{pzhang@mpim-bonn.mpg.de}
\date{03 November 2025}

\theoremstyle{plain}
\newtheorem{theorem}{Theorem}[section]

\newtheorem{lemma}[theorem]{Lemma}
\newtheorem{prop}[theorem]{Proposition}
\newtheorem{corol}[theorem]{Corollary}

\theoremstyle{remark}
\newtheorem{rmk}{Remark}[section]

\numberwithin{equation}{section}

\newcommand\m[1]{\begin{pmatrix}#1\end{pmatrix}}
\newcommand\bm[1]{\begin{bmatrix}#1\end{bmatrix}}
\newcommand\sm[1]{\begin{psmallmatrix}#1\end{psmallmatrix}}

\newcommand\mbf[1]{\mathbf{#1}}
\newcommand\cl[1]{\mathcal{#1}}
\newcommand\fr[1]{\mathfrak{#1}}
\newcommand\ov[1]{\overline{#1}}
\newcommand\wt[1]{\widetilde{#1}}
\newcommand\wh[1]{\widehat{#1}}

\newcommand{\rat}{\mathbb{Q}}
\newcommand{\inte}{\mathbb{Z}}
\newcommand{\real}{\mathbb{R}}
\newcommand{\comp}{\mathbb{C}}
\newcommand{\ade}{\mathbb{A}}
\newcommand{\pos}{\mathbb{Z}^+}
\newcommand{\eps}{\varepsilon}
\newcommand{\vphi}{\varphi}
\newcommand{\clh}{\mathcal{H}}
\newcommand{\clo}{\mathcal{O}}
\newcommand{\bslash}{\backslash}
\newcommand{\simarrow}{\xrightarrow{\;\sim\;}}
\newcommand{\eq}{\;=\;}
\newcommand{\deq}{\;:=\;}
\newcommand{\bone}{\mathbf{1}}
\def\thin{{\hskip 1pt}}
\newcommand{\hash}{\texttt{\#}}

\newcommand{\Hom}{\mathrm{Hom}}

\newcommand{\N}{\mathrm{N}}
\newcommand{\GL}{\mathrm{GL}}
\newcommand{\PGL}{\mathrm{PGL}}

\newcommand{\SL}{\mathrm{SL}}
\newcommand{\tr}{\mathrm{tr}}
\newcommand{\Gal}{\mathrm{Gal}}
\newcommand{\lcm}{\mathrm{lcm}}
\newcommand{\val}{\mathrm{val}}
\newcommand{\Cl}{\mathrm{Cl}}
\newcommand{\sgn}{\mathrm{sgn}}
\newcommand{\dif}{\mathrm{d}}
\renewcommand{\Re}{\operatorname{Re}}
\renewcommand{\Im}{\operatorname{Im}}

\newcommand{\SG}{\mathrm{S}\Gamma}

\begin{document}

\begin{abstract}
We prove two results on converse theorems for Hilbert modular forms over totally real fields of degree $r>1$. The first result recovers a Hilbert modular form (of some level) from an $L$-series satisfying functional equations twisted by all the unramified Hecke characters. The second result assumes both the `unramified' functional equations and an Euler product, and recovers a Hilbert modular form of the expected level predicted by the shape of the functional equations. Our result generalizes the current converse theorems for $\GL_2$ in the case of Hilbert modular forms in that we completely remove the assumptions on ramified twists.
\end{abstract}

\maketitle

\setcounter{tocdepth}{1}
\tableofcontents

\section{Introduction and the main result}

\subsection{Background and introduction}

The study of converse theorems, i.e., recovering automorphic forms from $L$-series, dates back to Hecke \cite{hecke}. He proved an equivalence between $L$-series satisfying a suitable functional equation and functions invariant under the element $\sm{0&-1\\1&0}$, which is then sufficient to recover elliptic modularity of level $1$. Hecke's work was generalized by Weil \cite{weil}, who recovered elliptic modularity of general level $N$ by assuming functional equations twisted by all the Dirichlet characters. Later in \cite{conreyfarmer}, Conrey and Farmer turned to another direction of generalization by imposing an Euler product to the $L$-series and recovered elliptic modularity for a few small levels. 

Converse theorems for Hilbert modular forms were first investigated by Doi--Naganuma \cite{doinaganuma} and Naganuma \cite{naganuma} when they discovered a lifting from elliptic modular forms to Hilbert modular forms, nowadays known as the Doi--Naganuma lifting. To justify the modularity of the image of the lifting, they proved a converse theorem for Hilbert modular forms of level~$1$, but subject to the base field being norm-Euclidean. In fact, subject to the class number $1$ condition, they proved an equivalence between $L$-series satisfying functional equations twisted by all the unramified Hecke characters and functions invariant under the element $\sm{0&-1\\1&0}$, but they could not justify that $\sm{0&-1\\1&0}$ and the translation matrices always generate the full Hilbert modular group. This generating result was only later proved by Vaserstein \cite{vaserstein}.

In the case of general number fields, it is well-known that Jacquet and Langlands \cite{jacquetlanglands} proved a converse theorem for $\GL_2$ in the style of Weil, i.e., assuming functional equations twisted by all the Hecke characters. 

The goal of this article is to recover Hilbert modularity from $L$-series satisfying
\begin{enumerate}[label=(\arabic*)]
    \item suitable functional equations twisted by all the \textbf{unramified} Hecke characters, and
    \item a suitable (partial) Euler product.
\end{enumerate}
We will prove that one can already recover a Hilbert modular form (of some level) by assuming all the `unramified' functional equations, and together with the Euler product, one can recover a Hilbert modular form of the expected level predicted by the shape of the functional equations. Our result can be viewed as a generalization or partial extension of many previous results, including \cite[Section~3]{doinaganuma}, \cite{p-s}, \cite{bookerkrish-strengthen}, \cite{bookerkrish-refine}, and \cite{bookerkrish-gln} (see \Cref{mainthmrmk} for a detailed discussion).

Roughly speaking, we will associate the $L$-series $L(s)$ with an $h$-tuple $\mbf{f}=(f_1,\ldots,f_h)$ of holomorphic functions on $\clh^r$ whose $L$-function equals $L(s)$, and prove that $\mbf{f}$ is a Hilbert modular form by showing that each component $f_\lambda$ is a classical Hilbert modular form (see \Cref{hilbert-modular-form-section} for the definitions). Here $h$ is the narrow class number of the base field and $r>1$ is the degree of the base field. The proof can be divided into four steps.
\begin{enumerate}[leftmargin=1.5cm, label=Step \arabic*.]
    \item Modify the proof in \cite[Section~3]{doinaganuma} to deduce from the `unramified' functional equations of $L(s)$ that $\mbf{f}$  is an eigenfunction of the `Fricke involution' (\Cref{fqnsection}).
    \item With $\mbf{f}$ now being an eigenfunction of the `Fricke involution', use the main theorem of Vaserstein in \cite{vaserstein} to deduce that each $f_\lambda$ is already invariant under a congruence subgroup of a larger level, which proves the first part of \Cref{mainthm} (\Cref{generating-sect,gamma1-inv}).
    \item Use the particular shape of the Euler factor of $L(s)$ at a certain prime to deduce that each $f_\lambda$ is an eigenfunction of the `Hecke operator at infinity' at that prime (\Cref{eulerprodsection}).
    \item With $f_\lambda$ now being a `Hecke eigenfunction' at those certain primes, use the idea from \cite[Theorem~2.2]{conjhecke} to extend the congruence subgroup that $f_\lambda$ is invariant under to the `correct' one predicted by the shape of the functional equations, which proves the second part of \Cref{mainthm} (\Cref{gamma0-inv-sect}).
\end{enumerate}

We would like to point out that there is a small but easily fixable gap in the proof of \cite[Theorem~2.2]{conjhecke}. We will also provide our own fix at the end of \Cref{gamma0-inv-sect}. For the coherence of the discussion, we postpone certain lemmas on prime ideals to the appendix.

\subsection{Setup and the main result}
\label{statement}
Let $F$ be a totally real field of degree $r>1$ and let $\clo_F$ be the ring of integers. Let $\sigma_1,\ldots,\sigma_r:F\rightarrow\real$ be the real embeddings of $F$. For $\alpha\in F$, write $\alpha^{(j)}=\sigma_j(\alpha)$, and $\alpha$ is called \textit{totally positive} if $\alpha^{(j)}>0$ for all $1\leq j\leq r$. For any set $E\subseteq F$, let $E_+$ denote the set of totally positive elements in~$E$. In particular, we write $\clo_+=(\clo_F)_+$ and $\clo^\times_+=(\clo_F^\times)_+$ for convenience. Let $\cl{I}_F$ denote the set of nonzero integral ideals of $\clo_F$ and let $\fr{d}$ denote the different of $F$. Let $\N$ denote both the norm map $\N_{F/\rat}$ on the elements of $F$ and the norm map on the (fractional) ideals of $F$. Let $\Cl^+(F)$ denote the narrow class group of $F$ and let $h=|\Cl^+(F)|$ denote the narrow class number of $F$. For any nonzero (fractional) ideal $\fr{a}$ of~$F$, let $[\fr{a}]$ denote the equivalence class of $\fr{a}$ in $\Cl^+(F)$. We will view each equivalence class in $\Cl^+(F)$ as a subset of $\cl{I}_F$. Let $\{\fr{t}_1,\ldots,\fr{t}_h\}\subseteq\cl{I}_F$ be a fixed set of representatives of $\Cl^+(F)$. 

We will give a brief discussion of Hecke characters and the classical formulation of Hilbert modular forms (of non-parallel weight) in \Cref{prelims}. Throughout the article, `unramified Hecke characters' mean Hecke characters that are unramified at every finite places. For convenience, we will use the notations in \Cref{prelims} without defining them here.

Let $\{A(\fr{a})\}_{\fr{a}\in \cl{I}_F}$ be a sequence of complex numbers indexed by nonzero integral ideals of $\clo_F$ such that $A(\clo_F)=1$ and $A(\fr{a})=O(\N(\fr{a})^c)$ for some constant $c>0$. Define an $L$-series
\begin{align}
\label{L-series-of-A-def}
    L(s)\eq L(s,A)\deq\sum_{\fr{a}\in\cl{I}_F}A(\fr{a})\cdot \N(\fr{a})^{-s}.
\end{align}
We will usually omit the $A$ in $L(s,A)$ and in the later notations when there is no confusion, as we think of $\{A(\fr{a})\}_{\fr{a}\in \cl{I}_F}$ as a fixed sequence.

For a nonzero integral ideal $\fr{m}$ of $\clo_F$, define
\begin{align*}
    K_0(\fr{m})&\deq\bigg\{\m{a&b\\c&d}\in\GL_2(\wh{\clo_F})\;\bigg|\;c\in\fr{m}\bigg\} \\
    K_1(\fr{m})&\deq\bigg\{\m{a&b\\c&d}\in\GL_2(\wh{\clo_F})\;\bigg|\;c\in\fr{m},\;d-1\in\fr{m}\bigg\},
\end{align*}
where $\wh{\clo_F}:=\varprojlim_{\fr{t}}\clo_F/\fr{t}=\prod_{\fr{p}}\clo_{F,\fr{p}}$.

Now, fix a nonzero integral ideal $\fr{n}$ of $\clo_F$ and an $r$-tuple $\mbf{k}=(k_1,\ldots,k_r)\in(2\inte^+)^r$ of positive even integers. Write $k_0=\max_{1\leq j\leq r}k_j$, $k_j^\prime=k_0-k_j$, and $\mbf{k}^\prime=(k_1^\prime,\ldots,k_r^\prime)$. For each unramified Hecke character~$\psi$ of $F$, define a twisted $L$-series
\begin{align*}
L(s,A,\psi)\deq\sum_{\fr{a}\in\cl{I}_F}\psi(\fr{a})\cdot A(\fr{a})\cdot \N(\fr{a})^{-s}
\end{align*}
and define the completed $L$-series as \footnote{One can \textit{a priori} define a completed $L$-series for any choices of $\fr{n}$ and $\mbf{k}$, but here $\fr{n}$ and $\mbf{k}$ should be viewed as part of the given data determined by the shape of the functional equations in \Cref{mainthm}.}
\begin{align*}
    \wh{L}(s,A,\psi)\deq(2\pi)^{-rs}\; \N(\fr{n}\fr{d}^2)^{s/2}\thin\prod_{j=1}^r\Gamma\bigg(s-\frac{k_j^\prime}{2}+\nu_j(\psi)\bigg)\thin L(s,A,\psi),
\end{align*}
where we will define $\nu_j(\psi)$ in \Cref{hecke-char-sect}.

\begin{theorem}
\label{mainthm}
Let $F$ be a totally real field of degree $r>1$. Let $\fr{n}$ be a nonzero integral ideal of~$\clo_F$, $\mbf{k}\in(2\inte^+)^r$, and $\epsilon=\pm1$. Let $\{\fr{t}_1,\ldots,\fr{t}_h\}$ be any set of representatives of $\Cl^+(F)$ and let $\fr{r}=\fr{t}_1\fr{t}_2\cdots\fr{t}_h$.

Suppose that for all the unramified Hecke characters $\psi$ of $F$, all the functions $\wh{L}(s,A,\psi)$ have holomorphic continuation to the whole complex plane which are bounded in vertical strips and satisfy functional equations
\begin{align*}
    \wh{L}(s,A,\psi)\eq \epsilon \thin i^{-|\mbf{k}|}\;\psi(\fr{n}\fr{d}^2)\;\wh{L}(k_0-s,A,\ov{\psi}),
\end{align*}
where $|\mbf{k}|=\sum_{j=1}^rk_j$. Then, $L(s)$ is the $L$-function of a Hilbert modular form of weight $\mbf{k}$ and level $K_1(\fr{r}\fr{n})$.

Suppose moreover that the $L$-series $L(s)$ has an Euler product of the form
\begin{align*}
    L(s)\eq\prod_{\fr{p}\text{ prime}}L_\fr{p}(s),
\end{align*}
where for $\fr{p}\nmid\fr{n}\fr{d}^2$,
\begin{align*}
    L_\fr{p}(s)\eq\big(1-A(\fr{p})\N(\fr{p})^s+\N(\fr{p})^{k_0-1-2s}\big)^{-1}.
\end{align*}
Then, $L(s)$ is the $L$-function of a Hilbert modular form of weight $\mbf{k}$ and level $K_0(\fr{n})$.
\end{theorem}
\begin{rmk}
\label{mainthmrmk}
We make several remarks on the theorem.
\begin{enumerate}[label=(\roman*)]
    \item If $F$ has narrow class number $1$, then $\fr{r}=\clo_F$ in the first part of the theorem, so all the `unramified' functional equations already recover a Hilbert modular form of the expected level. The two parts of the theorem can thus be viewed as a trade-off between having narrow class number~$1$ and having an Euler product.
    \item For the second part of the theorem, we in fact only need a partial Euler product at primes $\fr{p}$ that are of degree $1$, with $[\fr{p}]=[\clo_F]$ in $\Cl^+(F)$, and $\fr{p}\nmid\fr{n}\fr{d}^2$.
    \item This theorem should be viewed as a generalization of the converse theorem proved by Doi--Naganuma \cite{doinaganuma} and Naganuma \cite{naganuma}. Indeed, with our theorem, one may be able to follow the strategy of Doi--Naganuma to prove the Hilbert modularity of the image of the Doi--Naganuma lifting in more general cases.
    \item This theorem can be viewed as a partial extension of \cite{p-s}, \cite{bookerkrish-strengthen}, and \cite{bookerkrish-refine} when restricted to Hilbert modular forms over $F\neq\rat$, since our theorem removes the assumptions on ramified twists. The case when $F$ has narrow class number $1$ can be viewed as a special case of the main theorem in \cite{p-s}, though Piatetski-Shapiro's theorem is written in the language of automorphic representations. One should also turn to \cite[(iii),~p.671]{bookerkrish-strengthen}, \cite[Lemma~2.2]{bookerkrish-refine}, and \cite[Theorem~4.2]{cogdell-ps-work} for a more detailed discussion on Piatetski-Shapiro's theorem.
    \item In the case of $\GL_{n(\geq3)}$ over number fields, Cogdell and Piatetski-Shapiro \cite[Theorem~3]{cogdellpsi} first proved a converse theorem assuming only functional equations twisted by all the unramified automorphic representations of lower ranks, but subject to the class number $1$ condition. The class number $1$ restriction was later removed by Booker and Krishnamurthy \cite[Theorem~1.1]{bookerkrish-gln}. In this sense, our theorem may be viewed as a partial extension of this result from $\GL_{n(\geq3)}$ to $\GL_2$.
\end{enumerate}
\end{rmk}

\subsection{Acknowledgements}
The author would like to thank his advisor, Don Zagier, for all the suggestions regarding both mathematics and writing. The author would also like to thank Andrew Booker and Vesselin Dimitrov for useful remarks and correspondences as well as the anonymous referee(s) for their detailed advice. This paper is in addition dedicated to the memory of Ioan James, whose generous donation, in the form of the Ioan \& Rosemary James Scholarship, supported the author's MMath study at St John's College, Oxford. James was alive at the time the original version of this paper was completed, but sadly, he passed away on 21 February~2025.

\section{Preliminaries and notations}
\label{prelims}
In this section, we will review the theory of Hecke characters and the classical formulation of Hilbert modular forms. We will also write down a candidate for the Hilbert modular form that $L(s)$ should be associated with. Our main reference for Hecke characters is the book by Bump \cite{bump}, and the one for Hilbert modular forms is the paper by Shimura \cite{shimurahmf}, though in both cases there are some slight notational differences.

\subsection{Notations}

We first make a list of commonly used notations which we adopt for convenience. Let $\clh=\{z\in\comp\mid\Im(z)>0\}$ be the upper half plane. Define
\begin{align*}
    \GL_2^+(F)\deq\{\gamma\in\GL_2(F)\mid(\det\gamma)^{(j)}>0\text{ for all }1\leq j\leq r\}.
\end{align*}
Let $z=(z_1,\ldots,z_r)\in\clh^r$, $\kappa=(\kappa_1,\ldots,\kappa_r)\in\comp^r$, $\xi,\mu\in F$, and $\gamma\in\GL_2^+(F)$. We also denote $\gamma^{(j)}=\sm{a^{(j)}&b^{(j)}\\c^{(j)}&d^{(j)}}$ for $\gamma=\sm{a&b\\c&d}$.
We write
\begin{align*}
    \xi z+\mu&\eq(\xi^{(1)}z_1+\mu^{(1)},\ldots,\xi^{(r)}z_r+\mu^{(r)}) \\
    \gamma z&\eq(\gamma^{(1)}z_1,\ldots,\gamma^{(r)}z_r) \\
    z^\kappa&\eq\prod_{j=1}^rz_j^{\kappa_j},\hspace{0.5cm}\xi^\kappa\eq\prod_{j=1}^r(\xi^{(j)})^{\kappa_j} \\
    \tr(\xi z)&\eq\sum_{j=1}^r \xi^{(j)}z_j
\end{align*}
whenever they make sense. For example, $\xi^{\kappa}=\N(\xi)^k$ for $\kappa=(k,\ldots,k)$.

\subsection{Hecke characters}
\label{hecke-char-sect}
Let $\ade_F$ be the adele ring of $F$. A \emph{Hecke character of $F$} is a continuous group homomorphism $F^\times\bslash\ade_F^\times\rightarrow\comp^\times$. Let $\psi:F^\times\bslash\ade_F^\times\rightarrow\comp^\times$ be a Hecke character of $F$. For each place $v$ of $F$, define
\begin{align*}
    \psi_v:F_v^\times&\longrightarrow F^\times\backslash\mathbb{A}_F^\times\xrightarrow{\;\psi\;}\mathbb{C}^\times \\
    a\;&\mapstochar\xrightarrow{\hspace{2.5cm}}\psi(1,\ldots,1,a,1,\ldots,1),
\end{align*}
where $F_v$ is the completion of $F$ at the place $v$. Define also
\begin{align*}
    \psi_\infty:F_\infty\eq F\otimes_\rat\real\eq\prod_{j=1}^rF_{\sigma_j}&\longrightarrow\comp^\times \\
    (x_1,\ldots,x_r)&\longmapsto\prod_{j=1}^r\psi_{\sigma_j}(x_j),
\end{align*}
where $\sigma_j$ are the real places of $F$. It is known that at each real place $\sigma_j$ of $F$, there exist $\delta_j\in\{0,1\}$ and $\nu_j\in\comp$ such that
\begin{align*}
    \psi_{\sigma_j}(x^{(j)})\eq\sgn(x^{(j)})^{\delta_j}\thin|x^{(j)}|^{\nu_j}
\end{align*}
for all $x\in F$. Define $\nu_j(\psi)=\nu_j$ for each $1\leq j\leq r$ and write $\nu(\psi)=(\nu_1(\psi),\ldots,\nu_r(\psi))$.

Now, let $\psi$ be an \emph{unramified} Hecke character of $F$, i.e., $\psi_v(\clo_{F_v}^\times)=1$ for all finite places $v$ of~$F$. For a fractional ideal $\fr{a}$ of $F$, we may define
\begin{align*}
    \psi(\fr{a})\deq\prod_{i=1}^n\psi_{\fr{p}_i}(\varpi_{\fr{p}_i})^{e_i},
\end{align*}
where $\fr{a}=\prod_{i=1}^n\fr{p}_i^{e_i}$ is the unique factorization of $\fr{a}$ into prime ideals and $\varpi_{\fr{p}_i}$ is a uniformizer of~$F_{\fr{p}_i}$. This is well-defined since $\psi$ is unramified at all finite places.

It is worth noting that any character of the narrow class group of $F$, i.e., $\chi:\Cl^+(F)\rightarrow\comp^\times$ can be viewed as an unramified Hecke character of $F$. Specifically, for $(a_v)_v\in F^\times\backslash\ade_F^\times$, we may define
\begin{align*}
    \chi((a_v)_v)\deq\prod_{v=\fr{p}<\infty}\chi([\fr{p}])^{\mathrm{val}_\fr{p}(a_\fr{p})}.
\end{align*}
It is easy to check that this yields an unramified Hecke character of $F$, and in particular, $\nu(\chi)=(0,\ldots,0)$.

For the later purpose, we would like to construct a special family $(\psi_m)_{m\in\inte^{r-1}}$ of unramified Hecke characters of $F$. Fix a (multiplicative) basis $\{\eps_1,\ldots,\eps_{r-1}\}$ of~$\clo_+^\times$. For each \mbox{$m=(m_1,\ldots,m_{r-1})\in\inte^{r-1}$}, there exists (unique) \mbox{$\nu_m=(\nu_{m,1},\ldots,\nu_{m,r})\in (i\real)^r$} such that
\begin{equation}
\begin{aligned}
\label{nulinearrel}
    \sum_{j=1}^r\nu_{m,j}&\eq 0 \\
    \sum_{j=1}^r \nu_{m,j}\log (\eps_l^{(j)})&\eq 2\pi i m_l \hspace{0.5cm}\text{for }1\leq l\leq r-1.
\end{aligned}
\end{equation}
It is possible to make a choice of $(\delta_{m,j})_{j=1}^r\in\{0,1\}^r$ so that $\chi_m:F_\infty\rightarrow\comp$ given by
\begin{align*}
    \chi_m((x_j)_{j=1}^r)\deq\prod_{j=1}^r\sgn(x_j)^{\delta_{m,j}}\thin|x_j|^{\nu_{m,j}}
\end{align*}
defines a character on $\clo_F^\times\bslash F_\infty$ (see \cite[Section~1.7, p.~78]{bump}). Then, there exists a (not necessarily unique) Hecke character $\psi_m$ such that $(\psi_m)_{\infty}=\chi_m$ (see Exercise 5 in Section 6 of \cite[Chapter~VII]{neukirch}). In particular,
\begin{enumerate}[label=\rm{(\arabic*)}]
    \item $\nu_j(\psi_m)=\nu_{m,j}$ for all $1\leq j\leq r$;
    \item $\psi_m((\xi))=\prod_{j=1}^r(\xi^{(j)})^{-\nu_{m,j}}=\xi^{-\nu_m}$ for all totally positive $\xi\in F_+$.
\end{enumerate}
We will choose $(\psi_m)_{m\in\inte^{r-1}}$ so that $\psi_m\psi_{m^\prime}=\psi_{m+m^\prime}$. In particular, $\psi_{-m}=\psi_m^{-1}=\ov{\psi_m}$ and $\psi_{\mbf{0}}$ is the trivial Hecke character where $\mbf{0}=(0,\ldots,0)$. 

\subsection{Hilbert modular forms}
\label{hilbert-modular-form-section}

For $\gamma=\sm{a&b\\c&d}\in\GL_2^+(F)$ and \mbox{$\mbf{k}=(k_1,\ldots,k_r)\in(2\inte^+)^r$}, define the \textit{slash operator of weight $\mbf{k}$} on a holomorphic function $g:\clh^r\rightarrow\comp$ as
\begin{equation}
\begin{aligned}
\label{slashoperator}
    (g|_\mbf{k}\gamma)(z)&\deq\prod_{j=1}^r\big((\det\gamma^{(j)})^{k_j/2}\thin(c^{(j)}z_j+d^{(j)})^{-k_j}\big)\cdot g(\gamma^{(1)} z_1,\ldots,\gamma^{(r)}z_r) \\
    &\;\eq(\det\gamma)^{\mbf{k}/2}\cdot(cz+d)^{-\mbf{k}}\cdot g(\gamma z).
\end{aligned}
\end{equation}
Let $j:\GL_2^+(F)\rightarrow\PGL_2^+(F)$ be the projection map. A subgroup $\Gamma\subseteq\GL_2^+(F)$ is called an \textit{arithmetic} subgroup if $j(\Gamma\cap\GL_2^+(\clo_F))$ is a finite index subgroup of both $j(\Gamma)$ and $j(\GL_2^+(\clo_F))$. Let $\Gamma\subseteq\GL_2^+(F)$ be an arithmetic subgroup. A holomorphic function $g:\clh^r\rightarrow\comp$ is called a \textit{(classical) Hilbert modular form of weight $\mbf{k}$ and level $\Gamma$} if 
\begin{enumerate}[label=(\arabic*)]
    \item $g|_\mbf{k}\gamma=g$ for all $\gamma\in\Gamma$;
    \item for all $\gamma\in\GL_2^+(\clo_F)$, $g|_\mbf{k}\gamma$ is bounded as $\prod_{j=1}^r\Im(z_j)\rightarrow\infty$.
\end{enumerate}
It is well-known that when $r>1$, Koecher's principle implies that the second condition is not needed for the definition of Hilbert modular forms (see e.g.~\cite[Theorem~1.20]{123}).

We will be focused on adelic Hilbert modular forms on $K_0(\fr{n})$ and $K_1(\fr{n})$. We will not recall the definition of adelic Hilbert modular forms, but rather use the following equivalent viewpoint. For two nonzero integral ideals $\fr{t}$ and $\fr{m}$ of $\clo_F$, define
\begin{align*}
    \Gamma_0^+(\fr{t},\fr{m})&\deq\bigg\{\m{a&b\\c&d}\in\GL_2^+(F)\;\bigg|\;a\in\clo_F,\;b\in\fr{t}^{-1},\;c\in\fr{t}\fr{m},\;d\in\clo_F,\;ad-bc\in\clo_+^\times\bigg\} \\
    \Gamma_1^+(\fr{t},\fr{m})&\deq\bigg\{\m{a&b\\c&d}\in\GL_2^+(F)\;\bigg|\;a\in\clo_F,\;b\in\fr{t}^{-1},\;c\in\fr{t}\fr{m},\;d-1\in\fr{m},\;ad-bc\in\clo_+^\times\bigg\},
\end{align*}
where $\fr{t}^{-1}:=\{a\in F\thin|\thin(a)\fr{t}\subseteq\clo_F\}$. \footnote{The groups $\Gamma_i^+(\fr{t},\fr{m})$ are simply denoted as $\Gamma_i(\fr{t},\fr{m})$ in the standard notations. Here we add the $+$ sign to maintain consistency with the notations in \Cref{generating-sect}.} Then, an \emph{adelic Hilbert modular form of weight~$\mbf{k}$ and level $K_i(\fr{n})$} can be (equivalently) viewed as an $h$-tuple $\mbf{g}=\{g_1,\ldots,g_h\}$ where each $g_\lambda$ is a classical Hilbert modular form of weight $\mbf{k}$ and level $\Gamma_i^+(\fr{t}_\lambda,\fr{n})$ for $1\leq\lambda\leq h$. 

Let \mbox{$\mbf{g}=(g_1,\ldots,g_h)$} be an adelic Hilbert modular form of weight $\mbf{k}$ and level $K_i(\fr{n})$. Since $\Gamma_i^+(\fr{t}_\lambda,\fr{n})$ contains $T^\alpha=\sm{1&\alpha\\0&1}$ for all $\alpha\in\fr{t}_\lambda^{-1}$, the theory of Fourier series implies that $g_\lambda$ has a Fourier series expansion of the form
\begin{align*}
    g_\lambda(z)\eq b_\lambda(0)+\sum_{\xi\in(\fr{t}_\lambda\fr{d}^{-1})_+}b_\lambda(\xi)\thin e(\tr(\xi z)),
\end{align*}
where $e(x):=e^{2\pi ix}$. Here we recall that the \emph{different} $\fr{d}$ of $F$ is the integral ideal satisfying that
\begin{align*}
    \fr{d}^{-1}\eq\clo_F^\vee\deq\{x\in F\mid\tr(x\alpha)\in\inte\text{ for all }\alpha\in\clo_F\}.
\end{align*}
We will assume for simplicity that $b_\lambda(0)=0$ for all $1\leq\lambda\leq h$. Note that this condition is equivalent to $\mbf{g}$ being a cusp form in level $\fr{n}=1$ but weaker in general.

Since $\Gamma_i^+(\fr{t}_\lambda,\fr{n})$ also contains $E_\eps=\sm{\eps&0\\0&1}$ for all $\eps\in\clo_+^\times$, by comparing the Fourier coefficients, the condition that $g_\lambda|_\mbf{k}E_\eps=g_\lambda$ then implies that the value $b_\lambda(\xi)\xi^{\mbf{k}^\prime/2}$ only depends on the coset $\xi\clo_+^\times$, where we write $k_0=\max_{1\leq j\leq r}k_j$, $k_j^\prime=k_0-k_j$, and $\mbf{k}^\prime=(k_1^\prime,\ldots,k_r^\prime)$ as before. We thus define the $L$-function of $g_\lambda$ as
\begin{align*}
    L(s,g_\lambda)\deq\sum_{\xi\in(\fr{t}_\lambda\fr{d}^{-1})_+/\clo_+^\times}\N(\fr{t}_\lambda)^{k_0/2}\thin b_\lambda(\xi)\thin \xi^{\mbf{k}^\prime/2}\cdot\N(\xi\fr{t}_\lambda^{-1}\fr{d})^{-s}.
\end{align*}

Let $\cl{C}_\lambda:=[\fr{t}_\lambda^{-1}\fr{d}]$ denote an equivalence class in $\Cl^+(F)$. As every $\fr{a}\in\cl{C}_\lambda$ can be written as $\xi\fr{t}_\lambda^{-1}\fr{d}$ for some $\xi\in(\fr{t}_\lambda\fr{d}^{-1})_+$, we may define
\begin{align*}
    B(\fr{a})\eq B(\xi\fr{t}_\lambda^{-1}\fr{d})\deq\N(\fr{t}_\lambda)^{k_0/2}\thin b_\lambda(\xi)\thin\xi^{\mbf{k}^\prime/2}
\end{align*}
and rewrite the $L$-function of $g_\lambda$ as
\begin{align*}
    L(s,g_\lambda)\eq\sum_{\fr{a}\in\cl{C}_\lambda}B(\fr{a})\cdot\N(\fr{a})^{-s}.
\end{align*}
For each unramified Hecke character $\psi$ of $F$, define the twisted $L$-function as
\begin{align*}
    L(s,g_\lambda,\psi)\deq\sum_{\fr{a}\in\cl{C}_\lambda}\psi(\fr{a})\cdot B(\fr{a})\cdot\N(\fr{a})^{-s}
\end{align*}
and define the completed $L$-function as
\begin{align*}
    \wh{L}(s,g_\lambda,\psi)\deq(2\pi)^{-rs}\; \N(\fr{n}\fr{d}^2)^{s/2}\thin\prod_{j=1}^r\Gamma\bigg(s-\frac{k_j^\prime}{2}+\nu_j(\psi)\bigg)\thin L(s,g_\lambda,\psi).
\end{align*}
It follows from straightforward computations that $\wh{L}(s,g_\lambda,\psi)$ has the following integral representation
\begin{align}
\label{int-rep-L-function}
    \wh{L}(s,g_\lambda,\psi)\eq (2\pi)^{-|\mbf{k}^\prime|/2}\thin\psi(\fr{d})\thin\N(\fr{n})^{s/2}\int_{\clo^\times_+\bslash(\real_{>0})^r}g_\lambda(iy_1,\ldots,iy_r)\thin\prod_{j=1}^r y_j^{s-k'_j/2+\nu_j(\psi)}\thin\frac{\dif y_j}{y_j}.
\end{align}
From this, one deduces that $\wh{L}(s,g_\lambda,\psi)$ has a meromorphic continuation to the whole complex plane.

For each $1\leq\lambda\leq h$, let $\wt{\lambda}$ denote the unique integer such that $[\fr{t}_\lambda\fr{t}_{\wt{\lambda}}\fr{n}]=1$ in $\Cl^+(F)$. Let $q_\lambda$ be a totally positive generator of $\fr{t}_\lambda\fr{t}_{\wt{\lambda}}\fr{n}$ and define $W_{q_\lambda}=\sm{0&-1\\q_\lambda&0}$. It is easy to check that $g_\lambda|_\mbf{k}W_{q_\lambda}$ is now a Hilbert modular form of weight $\mbf{k}$ and level $\Gamma_i^+(\fr{t}_{\wt{\lambda}},\fr{n})$, and the integral representation (\ref{int-rep-L-function}) implies the following functional equation
\begin{align}
\label{classical-functional-equation}
    \wh{L}(s,g_\lambda,\psi)\eq i^{-|\mbf{k}|}\;\psi(\fr{n}\fr{d}^2)\;\wh{L}(k_0-s,g_\lambda|_\mbf{k}W_{q_\lambda},\ov{\psi}).
\end{align}

Now, define the $L$-function of $\mbf{g}=(g_1,\ldots,g_h)$ as
\begin{align*}
    L(s,\mbf{g})\deq\sum_{\fr{a}\in\cl{I}_F}B(\fr{a})\cdot\N(\fr{a})^{-s}\eq\sum_{\lambda=1}^hL(s,g_\lambda).
\end{align*}
For each unramified Hecke character $\psi$ of $F$, define the twisted $L$-function as
\begin{align*}
    L(s,\mbf{g},\psi)\deq\sum_{\fr{a}\in\cl{I}_F}\psi(\fr{a})\cdot B(\fr{a})\cdot\N(\fr{a})^{-s}\eq\sum_{\lambda=1}^hL(s,g_\lambda,\psi)
\end{align*}
and define the completed $L$-function as
\begin{align}
\label{adelic-equal-sum-of-classical}
    \wh{L}(s,\mbf{g},\psi)\deq(2\pi)^{-rs}\; \N(\fr{n}\fr{d}^2)^{s/2}\thin\prod_{j=1}^r\Gamma\bigg(s-\frac{k_j^\prime}{2}+\nu_j(\psi)\bigg)\thin L(s,\mbf{g},\psi)\eq\sum_{\lambda=1}^h\wh{L}(s,g_\lambda,\psi).
\end{align}
Then, $\wh{L}(s,\mbf{g},\psi)$ also has a meromorphic continuation to the whole complex plane.

To obtain the functional equation, suppose that $\mbf{g}$ is an eigenfunction of the Fricke involution, or equivalently, there exists some $\epsilon=\pm1$ such that $g_\lambda|_\mbf{k}W_{q_\lambda}=\epsilon g_{\wt{\lambda}}$ for all $1\leq\lambda\leq h$ (see \cite[pp.~654--655]{shimurahmf}). Then, it follows from Equations (\ref{classical-functional-equation}) and (\ref{adelic-equal-sum-of-classical}) that the completed $L$-function satisfies the functional equation
\begin{align*}
    \wh{L}(s,\mbf{g},\psi)\eq \epsilon\thin i^{-|\mbf{k}|}\;\psi(\fr{n}\fr{d}^2)\;\wh{L}(k_0-s,\mbf{g},\ov{\psi}).
\end{align*}
This motivates the particular shape of functional equations considered in \Cref{mainthm}.

For the Euler product, it is closely related to the Hecke operators as in the case of elliptic modular forms. However, the definition of Hecke operators on adelic Hilbert modular forms is quite technical and complicated. We will thus not recall the Hecke theory but rather refer to \cite[Section~2]{shimurahmf} for a detailed discussion. Suppose that $\mbf{g}$ is an eigenfunction of the Hecke operator at a prime $\fr{p}$ with $\fr{p}\nmid\fr{n}\fr{d}^2$. Then, the $L$-function of $\mbf{g}$ has a $\fr{p}$-partial Euler product of the form
\begin{align*}
    L(s,\mbf{g})\eq L^\fr{p}(s,\mbf{g})\;L_\fr{p}(s,\mbf{g})
\end{align*}
with
\begin{align*}
   L_\fr{p}(s,\mbf{g})\eq\big(1-B(\fr{p})\N(\fr{p})^s+\N(\fr{p})^{k_0-1-2s}\big)^{-1}\hspace{0.5cm}\text{and}\hspace{0.5cm}
    L^\fr{p}(s,\mbf{g})\eq\sum_{\fr{p}\nmid\fr{a}}B(\fr{a})\cdot\N(\fr{a})^{-s}.
\end{align*}
This motivates the particular shape of the Euler product considered in \Cref{mainthm}.

\subsection{A candidate for the Hilbert modular form}

Now, we go back to the setting of \Cref{mainthm}. Starting with the sequence $\{A(\fr{a})\}_{\fr{a}\subseteq\cl{O}_F}$, for each $1\leq \lambda\leq h$, we define
\begin{align*}
    a_\lambda(\xi)\deq\N(\fr{t}_\lambda)^{k_0/2}\thin A(\xi\fr{t}_\lambda^{-1}\fr{d})\thin \xi^{-\mbf{k}^\prime/2}
\end{align*}
for $\xi\in(\fr{t}_\lambda\fr{d}^{-1})_+$, and define a holomorphic function $f_\lambda:\clh^r\rightarrow\comp$ by
\begin{align*}
    f_\lambda(z)\deq\sum_{\xi\in(\fr{t}_\lambda\fr{d}^{-1})_+}a_\lambda(\xi)\thin e(\tr(\xi z)).
\end{align*}
Now, define an $h$-tuple $\mbf{f}=(f_1,\ldots,f_h)$. The following lemma follows easily from the previous review on Hilbert modular forms.

\begin{lemma}
\label{fhmf}
If $\mbf{f}=(f_1,\ldots,f_h)$ is an adelic Hilbert modular form of weight $\mbf{k}$ and level $K_i(\fr{m})$ for some nonzero ideal $\fr{m}$ of $\clo_F$, then $L(s,\mbf{f})=L(s,A)$, and in particular, $L(s,A)$ is an $L$-function of a Hilbert modular form of weight~$\mbf{k}$ and level~$K_i(\fr{m})$.
\end{lemma}

By \Cref{fhmf}, we will focus on proving that $\mbf{f}$ is an adelic Hilbert modular form from now on, i.e., each $f_\lambda$ is a classical Hilbert modular form of a certain level. We will finish this section with a basic lemma on $f_\lambda$. Recall that $T^\alpha=\sm{1&\alpha\\0&1}$ for $\alpha\in F$ and that $E_\eps=\sm{\eps&0\\0&1}$ for $\eps\in\clo_+^\times$.

\begin{lemma}
\label{basic}
For each $1\leq\lambda\leq h$, $f_\lambda|_\mbf{k}T^\alpha=f_\lambda$ for all $\alpha\in\fr{t}_\lambda^{-1}$ and $f_\lambda|_\mbf{k}E_\eps=f_\lambda$ for all  $\eps\in\clo_+^\times$.
\end{lemma}
\begin{proof}
It follows from the Fourier series expansion of $f_\lambda$ that $f_\lambda|_\mbf{k}T^\alpha=f_\lambda$ for all $\alpha\in\fr{t}_\lambda^{-1}$. The one for $E_\eps$ follows by observing that $a_\lambda(\xi)=\eps^{\mbf{k}^\prime/2}a_\lambda(\eps\xi)$ for all $\eps\in\clo_+^\times$ and $\xi\in(\fr{t}_\lambda\fr{d}^{-1})_+$.
\end{proof}

\section{Functional equations}
\label{fqnsection}

In this section, we will prove that all the `unramified' functional equations of $L(s)$ imply that $\mbf{f}$ is an eigenfunction of the `Fricke involution'. The result in this section is already known via the adelic language (see \cite[Proposition~10.2]{cogdellpsi} and \cite[Proposition~3.1]{bookerkrish-gln}), but a classical proof appears to be missing in the literature, so we would like to fill in one here. The proof is a direct generalization of \cite[Section~3]{doinaganuma}, where Doi and Naganuma proved the same result in the case of parallel weight, level~$1$, and class number $1$.
\\

For each unramified Hecke character $\psi$ and each equivalence class $\cl{C}\in\Cl^+(F)$, define a partial $L$-series
\begin{align*}
    L(s,\cl{C},\psi)\eq L(s,A,\cl{C},\psi)\deq\sum_{\fr{a}\in\cl{C}}\psi(\fr{a})\cdot A(\fr{a})\cdot\N(\fr{a})^{-s}.
\end{align*}
Indeed, for $\Re(s)$ sufficiently large,
\begin{align*}
    L(s,A,\psi)\eq\sum_{\cl{C}\in\Cl^+(F)}L(s,\cl{C},\psi).
\end{align*}
Define also the completed $L$-series as
\begin{align}
\label{completedpartiallseries}
    \wh{L}(s,\cl{C},\psi)\deq(2\pi)^{-rs}\; \N(\fr{n}\fr{d}^2)^{s/2}\thin\prod_{j=1}^r\Gamma\bigg(s-\frac{k_j^\prime}{2}+\nu_j(\psi)\bigg)\thin L(s,\cl{C},\psi).
\end{align}

\begin{lemma}
\label{partialfqn}
Let $\cl{C}\in\Cl^+(F)$ be an equivalence class and let $\wt{\cl{C}}=\cl{C}^{-1}\cdot[\fr{n}\fr{d}^2]\in\Cl^+(F)$. Suppose that all the $L$-series $L(s,A,\psi)$ satisfy the analytic properties and the functional equations as given in \Cref{mainthm}. Then, for all unramified Hecke characters $\psi$, all the functions $\wh{L}(s,\cl{C},\psi)$ have holomorphic continuation to the whole complex plane which are bounded in vertical strips and satisfy functional equations
\begin{align*}
    \wh{L}(s,\cl{C},\psi)\eq \epsilon \thin i^{-|\mbf{k}|}\thin\psi(\fr{n}\fr{d}^2)\thin\wh{L}(k_0-s,\wt{\cl{C}},\ov{\psi}).
\end{align*}
\end{lemma}
\begin{proof}
Let $G=\Cl^+(F)$ and let $\wh{G}=\Hom(\Cl^+(F),\comp^\times)$ be the dual group. By the theory of characters, $\wh{G}$ is a finite group and the indicator function $\delta_\cl{C}:G\rightarrow\comp$ can be written as $\delta_\cl{C}=\sum_{\chi\in\wh{G}}\alpha_\chi\cdot\chi$ for some $\alpha_\chi\in\comp$. Then, for $\Re(s)$ sufficiently large, it is easy to see that
\begin{align*}
    L(s,\cl{C},\psi)\eq\sum_{\fr{a}}\thin\delta_\cl{C}(\fr{a})\cdot\psi(\fr{a})\cdot A(\fr{a})\cdot\N(\fr{a})^{-s}\eq\sum_{\chi\in\wh{G}}\alpha_\chi L(s,A,\psi\chi),
\end{align*}
where we view each $\chi\in\wh{G}$ also as a Hecke character of $F$. Hence,
\begin{align*}
    \wh{L}(s,\cl{C},\psi)\eq\sum_{\chi\in\wh{G}}\alpha_\chi \wh{L}(s,A,\psi\chi).
\end{align*}
The analytic properties of $\wh{L}(s,A,\psi\chi)$ then imply the same for $\wh{L}(s,\cl{C},\psi)$. Now, it is easy to check that $\delta_{\wt{\cl{C}}}\eq\sum_{\chi\in\wh{G}}\alpha_\chi\thin\chi(\fr{n}\fr{d}^2)\thin\ov{\chi}$, so for $\Re(s)$ sufficiently large,
\begin{align*}
    \wh{L}(s,\wt{\cl{C}},\ov{\psi})\eq\sum_{\chi\in\wh{G}}\alpha_\chi\thin\chi(\fr{n}\fr{d}^2)\wh{L}(s,A,\ov{\psi\chi}),
\end{align*}
which also holds for all $s\in\comp$ after the holomorphic continuation. By applying the functional equations of $\wh{L}(s,A,\psi\chi)$, one thus obtains that
\begin{align*}
    \wh{L}(s,\cl{C},\psi)\eq\epsilon\thin i^{-|\mbf{k}|}\thin\psi(\fr{n}\fr{d}^2)\sum_{\chi\in\wh{G}}\alpha_\chi\thin\chi(\fr{n}\fr{d}^2)\wh{L}(k_0-s,A,\ov{\psi\chi})\eq\epsilon\thin i^{-|\mbf{k}|}\thin\psi(\fr{n}\fr{d}^2)\thin\wh{L}(k_0-s,\wt{\cl{C}},\ov{\psi}).
\end{align*}
\end{proof}

Recall that for each $1\leq\lambda\leq h$, $\cl{C}_\lambda=[\fr{t}_\lambda^{-1}\fr{d}]$ in $\Cl^+(F)$ and $\wt{\lambda}$ is the unique integer such that $[\fr{t}_\lambda\fr{t}_{\wt{\lambda}}\fr{n}]=[\clo_F]$ in $\Cl^+(F)$. Indeed, for each $1\leq\lambda\leq h$, $\wt{\cl{C}_\lambda}:=\cl{C}_\lambda^{-1}\cdot[\fr{n}\fr{d}^2]=\cl{C}_{\wt{\lambda}}$.

\begin{prop}
\label{fqnprop}
Fix $\lambda$ with $1\leq\lambda\leq h$ and let $\epsilon=\pm1$. Suppose that for any Hecke character $\psi_m$ defined in \Cref{hecke-char-sect}, both $\wh{L}(s,\cl{C}_\lambda,\psi_m)$ and $\wh{L}(s,\cl{C}_{\wt{\lambda}},\ov{\psi_m})$ have holomorphic continuation to the whole complex plane which are bounded in vertical strips and satisfy a functional equation
\begin{align}
\label{fqnequivclass}
    \wh{L}(s,\cl{C}_\lambda,\psi_m)\eq \epsilon \thin i^{-|\mbf{k}|}\thin\psi_m(\fr{n}\fr{d}^2)\thin\wh{L}(k_0-s,\cl{C}_{\wt{\lambda}},\ov{\psi_m}).
\end{align}
Then, $f_\lambda|_\mbf{k}\sm{0&-1\\q_\lambda&0}=\epsilon f_{\wt{\lambda}}$, where $q_\lambda$ is any totally positive generator of $\fr{t}_\lambda\fr{t}_{\wt{\lambda}}\fr{n}$.
\end{prop}

Before proceeding with the proof, we first make another list of notations that we adopt for convenience, which will mostly only appear in this section. Let $y=(y_1,\ldots,y_r)\in(\real_{>0})^r$, $\kappa=(\kappa_1,\ldots,\kappa_r)\in\comp^r$, $s\in\comp$, and $\xi\in F$. We write
\begin{align*}
    iy&\eq(iy_1,\ldots,iy_r),\hspace{0.5cm}\frac{i}{\xi y}\eq\bigg(\frac{i}{\xi^{(1)}y_1},\ldots,\frac{i}{\xi^{(r)}y_r}\bigg) \\
    y^s&\eq\prod_{j=1}^ry_j^s,\hspace{0.5cm} y^\kappa\eq\prod_{j=1}^ry_j^{\kappa_j},\hspace{0.5cm} y^{s+\kappa}\eq y^s\cdot y^\kappa, \hspace{0.5cm}
    \frac{\dif y}{y}\eq\prod_{j=1}^r\frac{\dif y_j}{y_j} \\
    \tr(\xi y)&\eq\sum_{j=1}^r \xi^{(j)}y_j, \hspace{0.5cm}\tr(\kappa y)\eq\sum_{j=1}^r\kappa_j y_j
\end{align*}
whenever they make sense. Throughout this section, we also adopt the notations $\psi_m$ and $\nu_m$ defined in \Cref{hecke-char-sect}. Note that this means we have already fixed a basis $\{\eps_1,\ldots,\eps_{r-1}\}$ of~$\clo_+^\times$. 
\\

Now, we return to the proof of \Cref{fqnprop}. We start by showing some lemmas, which essentially reproduces the work of Doi--Naganuma in \cite[Section~3]{doinaganuma} in a more general setting.

\begin{lemma}
\label{gammafunction}
Let $a=(a_j)\in(\real_{>0})^r$ and $b=(b_j)\in\comp^r$. Then, for $\sigma\in\real$ sufficiently large, 
\begin{align*}
    &\frac{1}{2\pi i}\int_{\sigma-i\infty}^{\sigma+i\infty}\prod_{j=1}^r a_j^{-s+b_j}\;\Gamma(s-b_j)\; \dif s\eq\\
    &\hspace{2cm}r^{-1}\;\Delta\int_{\real^{r-1}}\exp\bigg(-\sum\limits_{j=1}^r \big(a_j\prod\limits_{l=1}^{r-1}(\eps_l^{(j)})^{t_l}\big)\bigg)\;\prod_{j=1}^r\prod_{l=1}^{r-1}(\eps_l^{(j)})^{-t_lb_j}\;\dif t_1\cdots \dif t_{r-1},
\end{align*}
where
\begin{align*}
    \Delta\eq\begin{vmatrix}1&\log(\eps_1^{(1)})&\cdots&\log(\eps_{r-1}^{(1)})\\ 1&\log(\eps_1^{(2)})&\cdots&\log(\eps_{r-1}^{(2)})\\
    \vdots&\vdots & &\vdots \\ 1&\log(\eps_1^{(r)})&\cdots&\log(\eps_{r-1}^{(r)}) \end{vmatrix}.
\end{align*}
\end{lemma}
\begin{proof}
First, note that
\begin{align*}
    \prod_{j=1}^r a_j^{-s+b_j}\;\Gamma(s-b_j)\eq\prod_{j=1}^r\int_{\real_{>0}}e^{-a_jy_j}\; y_j^{s-b_j}\;\frac{\dif y_j}{y_j} \eq\int_{(\real_{>0})^r}e^{-\tr(ay)}\;y^{s-b}\;\frac{\dif y}{y}.
\end{align*}
We do a change of variables $(y_1,\ldots,y_r)\mapsto(y_0,t_1,\ldots,t_{r-1})$ as the following
\begin{align*}
    y_j\eq y_0\thin\prod_{l=1}^{r-1}(\eps_l^{(j)})^{t_l},
\end{align*}
where $y_0\in\real_{>0}$ and $t_l\in\real$. It is easy to check that
\begin{align*}
    y_1y_2\cdots y_r&\eq y_0^r \\
    \dif y_1\dif y_2\cdots \dif y_r&\eq \Delta\;y_0^{r-1}\;\dif y_0\;\dif t_1\cdots \dif t_{r-1}.
\end{align*}
Hence,
\begin{align*}
    \prod_{j=1}^r a_j^{-s+b_j}\;\Gamma(s-b_j)&\eq\int_{(\real_{>0})^r}e^{-\tr(ay)}\;y^{s-b}\;\frac{\dif y}{y} \\
    &\eq\Delta\int_{\real_{>0}}\bigg(\int_{\real^{r-1}}\exp\bigg(-\sum\limits_{j=1}^r a_jy_j\bigg)y^{-b}\;\dif t_1\cdots \dif t_{r-1}\bigg)y_0^{rs}\;\frac{\dif y_0}{y_0} \\
    &\eq r^{-1}\;\Delta\int_{\real_{>0}}\bigg(\int_{\real^{r-1}}\exp\bigg(-\sum\limits_{j=1}^r a_jy_j\bigg)y^{-b}\;\dif t_1\cdots \dif t_{r-1}\bigg)u^s\;\frac{\dif u}{u},
\end{align*}
where we still keep the variable $y_j=y_0\prod_{l=1}^{r-1}(\eps_l^{(j)})^{t_l}$ for convenience and do another change of variables $u=y_0^r$. By the inverse Mellin transform, we may pick $\sigma\in\real$ sufficiently large such that
\begin{align*}
    &\frac{1}{2\pi i}\int_{\sigma-i\infty}^{\sigma+i\infty}\prod_{j=1}^r a_j^{-s+b_j}\thin\Gamma(s-b_j)\thin u^{-s}\thin \dif s\eq r^{-1}\thin\Delta\int_{\real^{r-1}}\exp\bigg(-\sum\limits_{j=1}^r a_jy_j\bigg)y^{-b}\thin \dif t_1\cdots \dif t_{r-1},
\end{align*}
where we still keep the variable $y_j=u^{1/r}\prod_{l=1}^{r-1}(\eps_l^{(j)})^{t_l}$ for convenience. Plugging in $u=1$ on both sides then gives the result.
\end{proof}

Before stating the corollary, we write
\begin{align*}
    \eps(t)&\eq \prod_{l=1}^{r-1}\eps_l^{t_l} \\
    \eps(t)^c&\eq\prod_{j=1}^r\prod_{l=1}^{r-1}(\eps_l^{(j)})^{t_lc_j} \\
    \tr(a\eps(t))&\eq\sum_{j=1}^r \bigg(a_j\prod_{l=1}^{r-1}(\eps_l^{(j)})^{t_l}\bigg)
\end{align*}
for $t\in\real^{r-1}$, $c\in\real^r$, and $a\in\comp^r$ whenever they make sense.

\begin{corol}
\label{gammainversemellin-corol}
Fix $m\in\inte^{r-1}$ and let $\{\nu_{m,j}\}_{j=1}^r$ be defined as in \Cref{hecke-char-sect}. Let $\mbf{k}^\prime$ be as before. Let $a=(a_j)\in(\real_{>0})^r$. Then, 
\begin{equation}
\begin{aligned}
\label{gammainversemellin}
    &\frac{1}{2\pi i}\int_{\sigma-i\infty}^{\sigma+i\infty}\prod_{j=1}^r a_j^{-s+(k_j^\prime/2)-\nu_{m,j}}\;\Gamma\bigg(s-\frac{k_j^\prime}{2}+\nu_{m,j}\bigg)\thin \dif s\eq \\
    &\hspace{4cm}r^{-1}\;\Delta\int_{\real^{r-1}}e^{-\tr(a\eps(t))}\;\eps(t)^{-\mbf{k}^\prime/2}\thin e\bigg(\sum_{l=1}^{r-1}m_lt_l\bigg)\thin \dif t_1\cdots \dif t_{r-1}.
\end{aligned}
\end{equation}
\end{corol}
\begin{proof}
Apply \Cref{gammafunction} with $b_j=\frac{k_j^\prime}{2}-\nu_{m,j}$ and use that $\eps(t)^{\nu_m}=e\big(\sum_{l=1}^{r-1}m_lt_l\big)$ by \Cref{nulinearrel}.
\end{proof}

\begin{lemma}
\label{phibeta}
Let $\beta\in F_+$ be a totally positive element. Then,
\begin{equation}
\begin{aligned}
\label{phibetaeqn}
    \vphi_\beta(y)&\thin:=\thin\sum_{\eps\in\clo_+^\times}(\eps\beta)^{-\mbf{k}^\prime/2}\;e^{-2\pi\cdot\tr(\eps\beta y)} \\
    &\thin\;=\thin C\sum_{m\in\inte^{r-1}}\int_{\sigma-i\infty}^{\sigma+i\infty}(2\pi)^{-rs}\thin\N(\beta)^{-s}\thin\psi_m((\beta))\thin y^{-s+(\mbf{k}^\prime/2)-\nu_m}\prod_{j=1}^r\Gamma\bigg(s-\frac{k_j^\prime}{2}+\nu_{m,j}\bigg)\thin \dif s
\end{aligned}
\end{equation}
for $y\in(\real_{>0})^r$, where $C=\frac{r\Delta^{-1}(2\pi)^{|\mbf{k}^\prime|/2}}{2\pi i}$.
\end{lemma}
\begin{proof}
Note that
\begin{align*}
    \vphi_\beta(y)\eq\sum_{\eps\in\clo_+^\times}(\eps\beta)^{-\mbf{k}^\prime/2}\;e^{-2\pi\cdot\tr(\eps\beta y)}\eq\beta^{-\mbf{k}^\prime/2}\sum_{m\in\inte^{r-1}}\eps(m)^{-\mbf{k}^\prime/2}\;e^{-2\pi\cdot\tr(\eps(m)\beta y)}.
\end{align*}
Write
\begin{align*}
    S(t)\deq \eps(t)^{-\mbf{k}^\prime/2}\;e^{-2\pi\cdot\tr(\eps(t)\beta y)}\hspace{0.5cm}\text{and}\hspace{0.5cm}\wh{S}(x)\deq\int_{\real^{r-1}}S(t)\; e\bigg(-\sum_{l=1}^{r-1}x_lt_l\bigg)\thin \dif t
\end{align*}
for $t,x\in\real^{r-1}$. Then, applying \Cref{gammainversemellin-corol} with $a=(2\pi\beta^{(j)}y_j)_j\in(\real_{>0})^r$, we have
\begin{align*}
    \wh{S}(-m)&\eq\int_{\real^{r-1}}\eps(t)^{-\mbf{k}^\prime/2}\;e^{-2\pi\cdot\tr(\eps(t)\beta y)}\; e\bigg(\sum_{l=1}^{r-1}m_lt_l\bigg)\thin \dif t_1\cdots \dif t_{r-1} \\
    &\;\stackrel{(\ref{gammainversemellin})}{=}\;\frac{r\Delta^{-1}}{2\pi i}\int_{\sigma-i\infty}^{\sigma+i\infty}\prod_{j=1}^r(2\pi\beta^{(j)}y_j)^{-s+(k_j^\prime/2)-\nu_{m,j}}\;\Gamma\bigg(s-\frac{k_j^\prime}{2}+\nu_{m,j}\bigg)\thin \dif s
\end{align*}
for $m\in\inte^{r-1}$. By the Poisson summation formula,
\begin{align*}
    \vphi_\beta(y)&\eq\beta^{-\mbf{k}^\prime/2}\sum_{m\in\inte^{r-1}}S(m)\eq\beta^{-\mbf{k}^\prime/2}\sum_{m\in\inte^{r-1}}\wh{S}(m) \\
    &\eq\frac{r\Delta^{-1}\beta^{-\mbf{k}^\prime/2}}{2\pi i}\sum_{m\in\inte^{r-1}}\int_{\sigma-i\infty}^{\sigma+i\infty}\prod_{j=1}^r(2\pi\beta^{(j)}y_j)^{-s+(k_j^\prime/2)-\nu_{m,j}}\;\Gamma\bigg(s-\frac{k_j^\prime}{2}+\nu_{m,j}\bigg)\thin \dif s \\
    &\eq C\sum_{m\in\inte^{r-1}}\int_{\sigma-i\infty}^{\sigma+i\infty}(2\pi)^{-rs}\;\N(\beta)^{-s}\;\psi_m((\beta))\;y^{-s+(\mbf{k}^\prime/2)-\nu_m}\prod_{j=1}^r\Gamma\bigg(s-\frac{k_j^\prime}{2}+\nu_{m,j}\bigg)\thin \dif s,
\end{align*}
where $C=\frac{r\Delta^{-1}(2\pi)^{|\mbf{k}^\prime|/2}}{2\pi i}$.
\end{proof}

We are now ready to prove \Cref{fqnprop}. Applying \Cref{phibeta} with $\beta=\xi\in(\fr{t}_\lambda\fr{d}^{-1})_+$ and writing $\fr{a}=\xi\fr{t}_\lambda^{-1}\fr{d}\in\cl{C}_\lambda$, we obtain that
\begin{align*}
    &\vphi_{\xi}(y)\;\stackrel{(\ref{phibetaeqn})}{=}\; \\
    &\hspace{0.5cm}C\sum_{m\in\inte^{r-1}}\int_{\sigma-i\infty}^{\sigma+i\infty}(2\pi)^{-rs}\;\N(\fr{a}\fr{t}_\lambda\fr{d}^{-1})^{-s}\;\psi_m(\fr{a}\fr{t}_\lambda\fr{d}^{-1})\;y^{-s+(\mbf{k}^\prime/2)-\nu_m}\prod_{j=1}^r\Gamma\bigg(s-\frac{k_j^\prime}{2}+\nu_{m,j}\bigg)\thin \dif s.
\end{align*}
Hence,
\begin{align}
    f_\lambda(iy)&\eq\N(\fr{t}_\lambda)^{k_0/2}\sum_{\xi\in(\fr{t}_\lambda\fr{d}^{-1})_+/\clo_+^\times}\sum_{\eps\in\clo_+^\times}A(\xi\fr{t}_\lambda^{-1}\fr{d})\thin(\xi\eps)^{-\mbf{k}^\prime/2}\thin e^{-2\pi \cdot\tr(\xi\eps y)} \nonumber \\
    &\eq\N(\fr{t}_\lambda)^{k_0/2}\sum_{\xi\in(\fr{t}_\lambda\fr{d}^{-1})_+/\clo_+^\times}A(\xi\fr{t}_\lambda^{-1}\fr{d})\thin\vphi_\xi(y) \nonumber \\
    &\eq C\thin\N(\fr{t}_\lambda)^{k_0/2}\sum_{\fr{a}\in\cl{C}_\lambda}\sum_{m\in\inte^{r-1}}\int_{\sigma-i\infty}^{\sigma+i\infty}(2\pi)^{-rs}\;\;\N(\fr{a}\fr{t}_\lambda\fr{d}^{-1})^{-s}\;\psi_m(\fr{a}\fr{t}_\lambda\fr{d}^{-1})\;A(\fr{a}) \nonumber \\
    &\hspace{7cm}\times y^{-s+(\mbf{k}^\prime/2)-\nu_m}\;\prod_{j=1}^r\Gamma\bigg(s-\frac{k_j^\prime}{2}+\nu_{m,j}\bigg)\thin \dif s \nonumber \\
    &\;\stackrel{(\ref{completedpartiallseries})}{=}\; C\sum_{m\in\inte^{r-1}}\int_{\sigma-i\infty}^{\sigma+i\infty}\psi_{m}(\fr{t}_\lambda\fr{d}^{-1})\;\N(\fr{n})^{-s/2}\;\N(\fr{t}_\lambda)^{(k_0/2)-s}\;y^{-s+(\mbf{k}^\prime/2)-\nu_m}\;\wh{L}(s,\cl{C}_\lambda,\psi_m)\; \dif s. \label{flambdaiy}
\end{align}
By the functional equations,
\begin{align*}
    &f_\lambda(iy)\;\stackrel{(\ref{fqnequivclass})}{=}\; \\
    &\hspace{0.5cm}\epsilon\thin i^{-|\mbf{k}|}\thin C\sum_{m\in\inte^{r-1}}\int_{\sigma-i\infty}^{\sigma+i\infty}\psi_m(\fr{n}\fr{t}_\lambda\fr{d})\;\N(\fr{n})^{-s/2}\;\N(\fr{t}_\lambda)^{(k_0/2)-s}\; y^{-s+(\mbf{k}^\prime/2)-\nu_m}\;\wh{L}(k_0-s,\cl{C}_{\wt{\lambda}},\psi_{-m})\; \dif s.
\end{align*}
By moving the line of the integral (which is possible under the assumed analytic properties of all $\wh{L}(s,\cl{C}_\lambda,\psi_m)$), we obtain that
\begin{align*}
    &f_\lambda(iy)\eq\\
    &\hspace{0.5cm}\epsilon\thin C\thin (iy)^{-\mbf{k}}\sum_{m\in\inte^{r-1}}\int_{\sigma-i\infty}^{\sigma+i\infty}\psi_m(\fr{n}\fr{t}_\lambda\fr{d})\;\N(\fr{n})^{(s-k_0)/2}\;\N(\fr{t}_\lambda)^{s-(k_0/2)}\;y^{s-(\mbf{k}^\prime/2)-\nu_m}\;\wh{L}(s,\cl{C}_{\wt{\lambda}},\psi_{-m})\; \dif s.
\end{align*}
On the other hand, applying the computations for $f_\lambda(iy)$ to $f_{\wt{\lambda}}(i/(q_\lambda y))$, we obtain that
\begin{align*}
    &\epsilon\thin q_\lambda^{-\mbf{k}/2}\thin (iy)^{-\mbf{k}}\thin f_{\wt{\lambda}}\bigg(\frac{i}{q_\lambda y}\bigg)\\
    \;\stackrel{(\ref{flambdaiy})}{=}\;&\epsilon\thin C\thin q_\lambda^{-\mbf{k}/2}\thin(iy)^{-\mbf{k}}\sum_{m\in\inte^{r-1}}\int_{\sigma-i\infty}^{\sigma+i\infty}\psi_{m}(\fr{t}_{\wt{\lambda}}\fr{d}^{-1})\;\N(\fr{n})^{-s/2}\;\N(\fr{t}_{\wt{\lambda}})^{(k_0/2)-s} \\
    &\hspace{8cm}\times(1/q_\lambda y)^{-s+(\mbf{k}^\prime/2)-\nu_m}\;\wh{L}(s,\cl{C}_{\wt{\lambda}},\psi_m)\; \dif s \\
    \eq&\epsilon\thin C\thin(iy)^{-\mbf{k}}\sum_{m\in\inte^{r-1}}\int_{\sigma-i\infty}^{\sigma+i\infty}\psi_{m}(q_\lambda^{-1}\fr{t}_{\wt{\lambda}}\fr{d}^{-1})\;\N(\fr{n})^{-s/2}\;\N(\fr{t}_{\wt{\lambda}})^{(k_0/2)-s}\;\N(q_\lambda)^{s-(k_0/2)} \\
    &\hspace{8cm}\times y^{s-(\mbf{k}^\prime/2)+\nu_m}\;\wh{L}(s,\cl{C}_{\wt{\lambda}},\psi_m)\; \dif s \\
    \eq&\epsilon\thin C\thin(iy)^{-\mbf{k}}\sum_{m\in\inte^{r-1}}\int_{\sigma-i\infty}^{\sigma+i\infty}\psi_{-m}(\fr{n}\fr{t}_\lambda\fr{d})\;\N(\fr{n})^{(s-k_0)/2}\;\N(\fr{t}_\lambda)^{s-(k_0/2)}\;y^{s-(\mbf{k}^\prime/2)+\nu_m}\;\wh{L}(s,\cl{C}_{\wt{\lambda}},\psi_m)\; \dif s,
\end{align*}
where we use that $(q_\lambda)=\fr{t}_\lambda\fr{t}_{\wt{\lambda}}\fr{n}$. Hence, by comparing the two equations, we obtain that
\begin{align*}
    f_\lambda(iy)\eq\epsilon\thin q_\lambda^{-\mbf{k}/2}\thin (iy)^{-\mbf{k}}\thin f_{\wt{\lambda}}\bigg(\frac{i}{q_\lambda y}\bigg)
\end{align*}
for $y\in(\real_{>0})^r$. By analytic continuation, this implies that
\begin{align*}
    f_\lambda(z)\eq \epsilon\thin q_\lambda^{-\mbf{k}/2}\thin z^{-\mbf{k}}\thin f_{\wt{\lambda}}\bigg(-\frac{1}{q_\lambda z}\bigg)\eq \epsilon f_{\wt{\lambda}}|_\mbf{k}\m{0&-1\\q_\lambda&0}.
\end{align*}
As $\mbf{k}\in(2\inte^+)$, this is equivalent to $f_\lambda|_\mbf{k}\sm{0&-1\\q_\lambda&0}=\epsilon f_{\wt{\lambda}}$.

\begin{corol}
\label{fqncorol}
If all the $L$-series $L(s,A,\psi)$ satisfy the analytic properties and the functional equations as given in \Cref{mainthm}, then 
\begin{enumerate}[label={\rm(\arabic*)}]
    \item $f_\lambda|_\mbf{k}W_{q_\lambda}=\epsilon f_{\wt{\lambda}}$ for all $1\leq \lambda\leq h$, where $W_{q_\lambda}:=\sm{0&-1\\q_\lambda&0}$ and $q_\lambda$ is any totally positive generator of $\fr{t}_\lambda\fr{t}_{\wt{\lambda}}\fr{n}$;
    \item $f_\lambda|_\mbf{k}A_\beta=f_\lambda$ for all $\beta\in\fr{t}_\lambda\fr{n}$, where $A_\beta:=\sm{1&0\\\beta&1}$.
\end{enumerate}
\end{corol}
\begin{proof}
Here (1) follows from \Cref{partialfqn} and \Cref{fqnprop}. For (2), note that for any $\beta\in\fr{t}_\lambda\fr{n}$, we have $\beta/q_\lambda\in\fr{t}_{\wt{\lambda}}^{-1}$ and $A_\beta=W_{q_\lambda}^{-1}T^{-\beta/q_\lambda}W_{q_\lambda}$, so (2) then follows from an easy computation.
\end{proof}

\section{Generating sets of congruence subgroups}
\label{generating-sect}
In this section, we will recall the main theorem in \cite{vaserstein} and prove some important results on the generating sets of certain congruence subgroups, which will be used in the proof of \Cref{mainthm}. Throughout this section, we let $K$ denote a general number field to distinguish it from $F$, which we reserve for our totally real field of degree $r>1$.

We first recall the definition of Dedekind domains of arithmetic type from \cite{dedearith}, which was used in \cite{vaserstein}. Let $K$ be a number field and $\Sigma$ be a finite set of places of $K$ containing all the archimedean ones. Define
\begin{align*}
    \clo_{K,\Sigma}\deq\{\alpha\in K\mid \val_v(\alpha)\geq 0\text{ for all }v\notin \Sigma\}.
\end{align*}
A ring is called a \textit{Dedekind domain of arithmetic type} if it is of the form $\clo_{K,\Sigma}$ for some number field~$K$ and some finite set $\Sigma$ of places of $K$ containing all the archimedean ones.

Let $A=\clo_{K,\Sigma}$ be a Dedekind domain of arithmetic type for some $K$ and $\Sigma$ and let $I_1,I_2$ be two nonzero ideals of $A$. Define
\begin{align*}
    G(I_1,I_2)\deq\bigg\{\m{a&b\\c&d}\in\SL_2(A)\;\bigg|\; a-1,d-1\in I_1I_2,\;b\in I_1,\;c\in I_2\bigg\}
\end{align*}
and define $E(I_1,I_2)$ to be the subgroup of $\SL_2(A)$ generated by matrices of the form $\sm{1&b\\0&1}$ and $\sm{1&0\\c&1}$ for $b\in I_1$ and $c\in I_2$. Indeed, $E(I_1,I_2)\subseteq G(I_1,I_2)$, and $G(I_1,I_2)$ is a finite index subgroup of $\SL_2(A)$.

\begin{theorem}[Vaserstein \cite{vaserstein} \footnote{There is a small gap in Vaserstein's proof as the validity of \cite[Lemma~3]{vaserstein} was later called into question, but the proof was rewritten in \cite{liehl} without further controversy. One should see \cite{liehl} for more details.}]
\label{vaserstein}Let $A=\clo_{K,\Sigma}$ be a Dedekind domain of arithmetic type for some $K$ and $\Sigma$. Suppose that $A$ has infinitely many units, or equivalently, $\Sigma$ contains at least two places. Then, for any nonzero ideals $I_1,I_2$ of $A$, $E(I_1,I_2)$ is a normal and finite index subgroup of $G(I_1,I_2)$, and $E(A,A)=G(A,A)=\SL_2(A)$. Moreover, if $A$ is not an order in a totally imaginary number field, then $E(I_1,I_2)=G(I_1,I_2)$.
\end{theorem}
\begin{rmk}
A straightforward corollary is that if $K$ is a real quadratic field, then $\SL_2(\clo_K)$ is generated by matrices of the form $\sm{1&b\\0&1}$ and $\sm{1&0\\c&1}$ for \mbox{$b,c\in\clo_K$}. As mentioned in the introduction, together with the work by Doi--Naganuma \cite[Section~3]{doinaganuma}, this yields a converse theorem for Hilbert modular forms of level $1$, subject to the class number $1$ condition.
\end{rmk}

Now, we would like to apply \Cref{vaserstein} to our setting. We will be writing results in terms of general number fields so that they may be applied to more general cases. Let $K$ be a general number field. Define
\begin{align*}
    \GL_2^+(K)\deq\{\gamma\in\GL_2(K)\mid\sigma(\det\gamma)>0\text{ for all }\sigma:K\hookrightarrow\real\}.
\end{align*}
For nonzero ideals $\fr{t}$ and $\fr{m}$ of $\clo_K$, define
\begin{align*}
    \SG_1(\fr{m})&\deq\bigg\{\m{a&b\\c&d}\in\SL_2(\clo_K)\;\bigg|\; \m{a&b\\c&d}\equiv\m{1&*\\0&1}\pmod{\fr{m}}\bigg\} \\
    \Gamma_1(\fr{t},\fr{m})&\deq\bigg\{\m{a&b\\c&d}\in\GL_2(K)\;\bigg|\; a\in\clo_K,\;b\in\fr{t}^{-1},\;c\in\fr{t}\fr{m},\;d-1\in\fr{m},\;ad-bc\in\clo_K^\times\bigg\} \\
    \Gamma_1^+(\fr{t},\fr{m})&\deq\bigg\{\m{a&b\\c&d}\in\GL_2(K)\;\bigg|\; a\in\clo_K,\;b\in\fr{t}^{-1},\;c\in\fr{t}\fr{m},\;d-1\in\fr{m},\;ad-bc\in(\clo_K^\times)_+\bigg\}\\
    &\;\eq\Gamma_1(\fr{t},\fr{m})\cap\GL_2^+(K).
\end{align*}
Here $(\clo_K^\times)_+$ denotes the set of all totally positive units, i.e., $\eps\in\clo_K^\times$ such that $\sigma(\eps)>0$ for all $\sigma:K\hookrightarrow\real$.

\begin{lemma}
\label{sgamma1-generator}
Let $K\neq\rat$ be a number field that has at least one real place and let $\fr{m}$ be a nonzero ideal of $\clo_K$. Then, $\SG_1(\fr{m})$ is generated by matrices of the form $\sm{1&b\\0&1}$ and $\sm{1&0\\c&1}$ for $b\in\clo_K$ and $c\in\fr{m}$.
\end{lemma}
\begin{proof}
By the assumption on $K$, $K$ has at least two archimedean places and $\clo_K$ is not an order in a totally imaginary number field. Hence, we may apply \Cref{vaserstein} with $A=\clo_K$, $I_1=\clo_K$, and $I_2=\fr{m}$ to obtain the result.
\end{proof}

\begin{prop}
\label{gamma1-generator}
Let $K\neq\rat$ be a number field that has at least one real place and let $\fr{t}$ and $\fr{m}$ be nonzero ideals of $\clo_K$. Then, $\Gamma_1(\fr{t},\fr{tm})$ (resp. $\Gamma_1^+(\fr{t},\fr{tm})$) is contained in the group generated by matrices of the form $\sm{1&b\\0&1}$, $\sm{1&0\\c&1}$, and $\sm{\eps&0\\0&1}$ for $b\in\fr{t}^{-1}$, $c\in\fr{tm}$, and $\eps\in\clo_K^\times$ (resp. $\eps\in(\clo_K^\times)_+$).
\end{prop}
\begin{proof}
Let $\gamma\in\Gamma_1(\fr{t},\fr{tm})$ or $\Gamma_1^+(\fr{t},\fr{tm})$. If we write $\eps=\det\gamma$, then we may replace $\gamma$ by $\sm{\eps^{-1}&0\\0&1}\gamma$ to assume that $\det\gamma=1$. In this way, it suffices to show that any $\gamma\in\Gamma_1(\fr{t},\fr{tm})$ with $\det\gamma=1$ is generated by matrices of the form $\sm{1&\alpha\\0&1}$ and $\sm{1&0\\\beta&1}$ for $\alpha\in\fr{t}^{-1}$ and $\beta\in\fr{tm}$.

Write $\gamma=\sm{a&b\\c&d}$. Note that the ideals $(a)$ and $\fr{t}$ must be coprime, as otherwise their common prime factor would divide $(ad-bc)$ (since $\fr{tm}|(bc)$), contradicting that $ad-bc=1$. Now, let $S$ be a set of representatives of $\clo_K/\fr{t}$. Since $(a)$ and $\fr{t}$ are coprime, $\{an+1|n\in S\}$ is also a set of representatives of $\clo_K/\fr{t}$. In particular, there exists $n\in\clo_K$ such that $an+1\in\fr{t}$. Note that
\begin{align*}
    \gamma\m{1&bn\\0&1}\eq\m{a&b(an+1)\\c&bcn+d}.
\end{align*}
Here, we have
\begin{enumerate}[label=(\arabic*)]
    \item $a\in\clo_K$;
    \item $b(an+1)\in\clo_K$ since $b\in\fr{t}^{-1}$ and $an+1\in\fr{t}$;
    \item $c\in\fr{t}^2\fr{m}\subseteq\fr{tm}$;
    \item $bcn+d-1\in\fr{tm}$ since $b\in\fr{t}^{-1}$, $c\in\fr{t}^2\fr{m}$, and $d-1\in\fr{tm}$.
\end{enumerate}
As $ad-bc=1$, we also have $a-1\in\fr{tm}$. Hence, $\gamma\sm{1&bn\\0&1}\in\SG_1(\fr{tm})$. The result then follows from \Cref{sgamma1-generator}.
\end{proof}

We will also prove a result for $\Gamma_0$ for the later purpose, the proof of which does not rely on \Cref{vaserstein} though. Define
\begin{align*}
    \SG_0(\fr{m})&\deq\bigg\{\m{a&b\\c&d}\in\SL_2(\clo_K)\;\bigg|\; c\in\fr{m}\bigg\} \\
    \Gamma_0(\fr{t},\fr{m})&\deq\bigg\{\m{a&b\\c&d}\in\GL_2(K)\;\bigg|\; a\in\clo_K,\;b\in\fr{t}^{-1},\;c\in\fr{t}\fr{m},\;d\in\clo_K,\;ad-bc\in\clo_K^\times\bigg\} \\
    \Gamma_0^+(\fr{t},\fr{m})&\deq\bigg\{\m{a&b\\c&d}\in\GL_2(K)\;\bigg|\; a\in\clo_K,\;b\in\fr{t}^{-1},\;c\in\fr{t}\fr{m},\;d\in\clo_K,\;ad-bc\in(\clo_K^\times)_+\bigg\}\\
    &\;\eq\Gamma_0(\fr{t},\fr{m})\cap\GL_2^+(K).
\end{align*}

\begin{prop}
\label{gamma0-generator}
Let $K$ be a number field. Let $\fr{t}$ and $\fr{m}$ be nonzero ideals of $\clo_K$ such that $\fr{t}$ is a \textbf{prime} ideal. Then, $\Gamma_0(\fr{t},\fr{m})$ (resp. $\Gamma_0^+(\fr{t},\fr{m})$) is generated by $\SG_0(\fr{tm})$ and matrices of the form $\sm{1&b\\0&1}$ and $\sm{\eps&0\\0&1}$ for $b\in\fr{t}^{-1}$ and $\eps\in\clo_K^\times$ (resp. $\eps\in(\clo_K^\times)_+$).
\end{prop}
\begin{proof}
Similarly as in the proof of \Cref{gamma1-generator}, it suffices to show that any $\gamma=\sm{a&b\\c&d}\in\Gamma_0(\fr{t},\fr{m})$ with $\det\gamma=1$ is generated by $\SG_0(\fr{tm})$ and matrices of the form $\sm{1&\alpha\\0&1}$ for $\alpha\in\fr{t}^{-1}$.

First, suppose that $a\in\fr{t}$. Let $\fr{q}\neq\fr{t}$ be another prime ideal with $[\fr{q}]=[\fr{t}]$ in the class group $\Cl(K)$ of $K$ (such $\fr{q}$ exists by Chebotarev's density theorem). Then, there exists $l\in K$ such that $(l)\fr{t}=\fr{q}$ (and this also implies that $l\in\fr{t}^{-1}$). Suppose for contradiction that $a+cl\in\fr{t}$. Then, $cl\in\fr{t}$ since $a\in\fr{t}$. Note that $(cl)=(c)\fr{t}^{-1}\fr{q}$, where $(c)\fr{t}^{-1}$ is an integral ideal since $c\in\fr{tm}$. As $\fr{t}|(cl)$ and $(\fr{t},\fr{q})=1$, this implies that $\fr{t}|(c)\fr{t}^{-1}$. Thus, $\fr{t}|(bc)$ since $b\in\fr{t}^{-1}$ and hence $\fr{t}|(ad-bc)$ since $a\in\fr{t}$ by the assumption, which leads to contradiction as $\fr{t}$ is a prime ideal while $ad-bc=\det\gamma=1$. Hence, $a+cl\notin\fr{t}$, so we may replace $\gamma$ by $\sm{1&l\\0&1}\gamma$ to assume that $a\notin\fr{t}$.

Now, $(a)$ and $\fr{t}$ are coprime, so similarly as in the proof of \Cref{gamma1-generator}, there exists $n\in\clo_K$ such that $an+1\in\fr{t}$. Then,
\begin{align*}
    \gamma\m{1&bn\\0&1}\eq\m{a&b(an+1)\\c&bcn+d}
\end{align*}
with
\begin{enumerate}[label=(\arabic*)]
    \item $a\in\clo_K$;
    \item $b(an+1)\in\clo_K$ since $b\in\fr{t}^{-1}$ and $an+1\in\fr{t}$;
    \item $c\in\fr{tm}$;
    \item $bcn+d\in\clo_K$ since $b\in\fr{t}^{-1}$, $c\in\fr{tm}$, and $d\in\clo_K$.
\end{enumerate}
Hence, $\gamma\sm{1&bn\\0&1}\in\SG_0(\fr{tm})$. The result then follows.
\end{proof}

\section{Invariance of $f_\lambda$ under $\Gamma_1^+(\fr{t}_\lambda,\fr{r}\fr{n})$}
\label{gamma1-inv}

We are now ready to give a proof of the first part of \Cref{mainthm}, i.e., recovering a Hilbert modular form assuming only the `unramified' functional equations.

\begin{prop}
\label{f_lambda-gamma1-inv}
Let $\{\fr{t}_1,\ldots,\fr{t}_h\}$ be a set of representatives of $\Cl^+(F)$. If all the $L$-series $L(s,A,\psi)$ satisfy the analytic properties and the functional equations as given in \Cref{mainthm}, then for each \mbox{$1\leq\lambda\leq h$}, the associated function $f_\lambda$ satisfies that $f_\lambda|_\mbf{k}\gamma=f_\lambda$ for all $\gamma\in\Gamma_1^+(\fr{t}_\lambda,\fr{t}_\lambda\fr{n})$.
\end{prop}
\begin{proof}
By \Cref{basic} and \Cref{fqncorol}, we have that for all $1\leq\lambda\leq h$,
\begin{enumerate}[label=\rm{(\alph*)}]
    \item $f_\lambda|_\mbf{k}T^\alpha=f_\lambda$ for all $\alpha\in\fr{t}_\lambda^{-1}$, where $T^\alpha=\sm{1&\alpha\\0&1}$;
    \item $f_\lambda|_\mbf{k}E_\eps=f_\lambda$ for all $\eps\in\clo_+^\times$, where $E_\eps=\sm{\eps&0\\0&1}$;
    \item $f_\lambda|_\mbf{k}A_\beta=f_\lambda$ for all $\beta\in\fr{t}_\lambda\fr{n}$, where $A_\beta=\sm{1&0\\\beta&1}$. 
\end{enumerate}
As $F$ is a totally real field of degree $r>1$, it follows from \Cref{gamma1-generator} that $f_\lambda$ is invariant under $\Gamma_1^+(\fr{t}_\lambda,\fr{t}_\lambda\fr{n})$. 
\end{proof}

\begin{corol}
\label{f-gamma1-inv}
Let $\{\fr{t}_1,\ldots,\fr{t}_h\}$ be a set of representatives of $\Cl^+(F)$ and let $\fr{r}=\fr{t}_1\cdots\fr{t}_h$. If all the $L$-series $L(s,A,\psi)$ satisfy the analytic properties and the functional equations as given in \Cref{mainthm}, then the associated $h$-tuple $\mbf{f}=(f_1,\ldots,f_h)$ is a Hilbert modular form of weight $\mbf{k}$ and level $K_1(\fr{rn})$.
\end{corol}
\begin{proof}
By \Cref{f_lambda-gamma1-inv}, each $f_\lambda$ is a classical Hilbert modular form of weight $\mbf{k}$ and level $\Gamma_1^+(\fr{t}_\lambda,\fr{t}_\lambda\fr{n})$, so in particular, it is also of level $\Gamma_1^+(\fr{t}_\lambda,\fr{r}\fr{n})$. The result then follows.
\end{proof}

The first part of \Cref{mainthm} now simply follows from \Cref{f-gamma1-inv} and \Cref{fhmf}. 

\section{The Euler product}
\label{eulerprodsection}
In this section, we will define the `Hecke operators at infinity', denoted as $T_{p,\cl{F}}^\infty$ (this notation will become clear later). We will prove some formalism for their actions on $f_\lambda$, and in the end, show that the shape of the Euler product of $L(s)$ implies that $f_\lambda$ is an eigenfunction of $T_{p,\cl{F}}^\infty$ for suitable $p$ and $\cl{F}$.

Throughout this section, we will fix the following data unless stated otherwise:
\begin{enumerate}[label=(\arabic*)]
    \item a fixed $\lambda$ with $1\leq \lambda\leq h$;
    \item a prime ideal $\fr{p}$ with $[\fr{p}]=[\clo_F]$ in $\Cl^+(F)$ and $\fr{p}\nmid\fr{n}\fr{d}^2$;
    \item a totally positive generator $p$ of $\fr{p}$;
    \item a set $\cl{F}$ of representatives of $\clo_F/\fr{p}$.
\end{enumerate}
From now on, we will always write `$(\fr{p},p,\cl{F})$ as before' whenever we mean a triple $(\fr{p},p,\cl{F})$ satisfying (2), (3), and (4) as above.

For a holomorphic function $g:\clh^r\rightarrow\comp$, define
\begin{align*}
    T_{p,\cl{F}}^\infty\thin g\deq \N(\fr{p})^{k_0/2-1}\bigg(g|_\mbf{k}\m{p&0\\0&1}+\sum_{\alpha\in\cl{F}}g|_\mbf{k}\m{1&\alpha\\0&p}\bigg).
\end{align*}
Explicitly, we have
\begin{align}
\label{explicithecke}
    (T_{p,\cl{F}}^\infty \thin g)(z)\eq p^{\mbf{k}+(\mbf{k}^\prime/2)-\bone}g(pz)+p^{\mbf{k}^\prime/2-\bone}\sum_{\alpha\in\cl{F}}g\big(\tfrac{z+\alpha}{p}\big),
\end{align}
where $\bone:=(1,1,\ldots,1)\in\inte^r$. 

Recall that $f_\lambda:\clh^r\rightarrow\comp$ is given by
\begin{align*}
    f_\lambda(z)\eq\sum_{\xi\in(\fr{t}_\lambda\fr{d}^{-1})_+}a_\lambda(\xi)\thin e(\tr(\xi z)),
\end{align*}
where $a_\lambda(\xi)=\N(\fr{t}_\lambda)^{k_0/2}A(\xi\fr{t}_\lambda^{-1}\fr{d})\xi^{-\mbf{k}^\prime/2}$ for $\xi\in(\fr{t}_\lambda\fr{d}^{-1})_+$. The following proposition gives the Fourier coefficients of $T_{p,\cl{F}}^\infty f_\lambda$ in terms of the Fourier coefficients of $f_\lambda$, which is analogous to the case of elliptic modular forms (see e.g. \cite[Proposition~5.2.2]{diamond-shurman}).

\begin{prop}
The Fourier series expansion of $T_{p,\cl{F}}^\infty f_\lambda$ is
\begin{align*}
    (T_{p,\cl{F}}^\infty\thin f_\lambda)(z)\eq\sum_{\xi\in(\fr{t}_\lambda\fr{d}^{-1})_+}b(\xi)\thin e(\tr(\xi z)),
\end{align*}
where
\begin{align*}
    b(\xi)\eq p^{\mbf{k}^\prime/2}\thin \big(a_\lambda(\xi p)+p^{\mbf{k}-\bone}\thin a_\lambda\big(\tfrac{\xi}{p}\big)\big)
\end{align*}
for $\xi\in(\fr{t}_\lambda\fr{d}^{-1})_+$, and $a_\lambda(\alpha):=0$ for $\alpha\notin(\fr{t}_\lambda\fr{d}^{-1})_+$ by convention.
\end{prop}
\begin{proof}
We have
\begin{align*}
    p^{\mbf{k}+(\mbf{k}^\prime/2)-\bone}f_\lambda(pz)\eq p^{\mbf{k}+(\mbf{k}^\prime/2)-\bone}\;\sum_{\mathclap{\xi\in(\fr{t}_\lambda\fr{d}^{-1})_+}}\;a_\lambda(\xi)\thin e(\tr(\xi pz))\eq p^{\mbf{k}+(\mbf{k}^\prime/2)-\bone}\;\sum_{\mathclap{\xi\in(\fr{t}_\lambda\fr{d}^{-1})_+}}\;a_\lambda\big(\tfrac{\xi}{p}\big)\thin e(\tr(\xi z))
\end{align*}
and
\begin{align*}
    p^{\mbf{k}^\prime/2-\bone}\sum_{\alpha\in\cl{F}}f\bigg(\frac{z+\alpha}{p}\bigg)&\eq p^{\mbf{k}^\prime/2-\bone}\sum_{\alpha\in\cl{F}}\sum_{\xi\in(\fr{t}_\lambda\fr{d}^{-1})_+}a_\lambda(\xi)\thin e\big(\tr\big(\tfrac{\xi(z+\alpha)}{p}\big)\big) \\
    &\eq p^{\mbf{k}^\prime/2-\bone}\sum_{\xi\in(\fr{t}_\lambda\fr{d}^{-1})_+}\bigg(\sum_{\alpha\in\cl{F}}e\big(\tr\big(\tfrac{\xi \alpha}{p}\big)\big)\bigg)\thin a_\lambda(\xi)\thin e\big(\tr\big(\tfrac{\xi z}{p}\big)\big) \\
    &\eq p^{\mbf{k}^\prime/2-\bone}\sum_{\substack{\xi\in(\fr{t}_\lambda\fr{d}^{-1})_+\\ \xi/p\in(\fr{t}_\lambda\fr{d}^{-1})_+}}|\clo_F/\fr{p}|\cdot a_\lambda(\xi)\thin e\big(\tr\big(\tfrac{\xi z}{p}\big)\big) \\
    &\eq p^{\mbf{k}^\prime/2} \sum_{\xi\in(\fr{t}_\lambda\fr{d}^{-1})_+}a_\lambda(\xi p)\thin e(\tr(\xi z)),
\end{align*}
where in the third line we use that
\begin{align*}
    \sum_{\alpha\in\cl{F}}e\big(\tr\big(\tfrac{\xi \alpha}{p}\big)\big)\eq\begin{cases} |\clo_F/\fr{p}|&\text{ if }\xi/p\in(\fr{t}_\lambda\fr{d}^{-1})_+; \\ 0&\text{ if }\xi/p\notin(\fr{t}_\lambda\fr{d}^{-1})_+, 
    \end{cases}
\end{align*}
which follows from viewing $\alpha\mapsto e(\tr(\xi\alpha/p))$ as a character on $\clo_F/\fr{p}$ since $\xi\in\fr{d}^{-1}$. The result then follows from \Cref{explicithecke}.
\end{proof}

\begin{prop}
\label{eigenfourier}
Let $C\in\comp$. Then, the following are equivalent:
\begin{enumerate}[label=\rm{(\roman*)}]
    \item $T^\infty_{p,\cl{F}} f_\lambda=Cf_\lambda$;
    \item $C\cdot A(\fr{a})=A(\fr{a}\fr{p})+\N(\fr{p})^{k_0-1}A(\fr{a}/\fr{p})$ for all nonzero integral ideals $\fr{a}\in\cl{C}_\lambda=[\fr{t}_\lambda^{-1}\fr{d}]$, where \mbox{$A(\fr{b}):=0$} if $\fr{b}$ is not an integral ideal by convention.
\end{enumerate}
\end{prop}
\begin{rmk}
Suppose moreover that the sequence $\{A(\fr{a})\}_{\fr{a}\subseteq\cl{O}_F}$ satisfies that $A(\fr{a})A(\fr{p})=A(\fr{ap})$ for all $\fr{a}$ with $\fr{p}\nmid\fr{a}$. Then, this proposition also implies that if $f_\lambda$ is a $T^\infty_{p,\cl{F}}$-eigenfunction, then its eigenvalue must be $A(\fr{p})$. This can be deduced by choosing some $\fr{a}\in\cl{C}_\lambda$ with $\fr{p}\nmid\fr{a}$ in (ii).
\end{rmk}
\begin{proof}
Write $b(\xi)$ as the Fourier coefficients of $T_{p,\cl{F}}^\infty f_\lambda$ as before. Note that (i) is equivalent to that for all $\xi\in(\fr{t}_\lambda\fr{d}^{-1})_+$,
\begin{align*}
    C\cdot a_\lambda(\xi)\eq b(\xi)\eq p^{\mbf{k}^\prime/2}\thin \big(a_\lambda(\xi p)+p^{\mbf{k}-\bone}\thin a_\lambda\big(\tfrac{\xi}{p}\big)\big).
\end{align*}
Using that
\begin{align*}
    a_\lambda(\xi)\eq\N(\fr{t}_\lambda)^{k_0/2}\thin A(\xi\fr{t}_\lambda^{-1}\fr{d})\thin \xi^{-\mbf{k}^\prime/2},
\end{align*}
this is equivalent to that
\begin{align*}
    C\cdot A(\xi\fr{t}_\lambda^{-1}\fr{d})\eq A\big((\xi\fr{t}_\lambda^{-1}\fr{d})\fr{p}\big)+\N(\fr{p})^{k_0-1}A\big((\xi\fr{t}_\lambda^{-1}\fr{d})/\fr{p}\big)
\end{align*}
for all $\xi\in(\fr{t}_\lambda\fr{d}^{-1})_+$. Finally, the latter is seen to be equivalent to (ii) by setting $\fr{a}=\xi\fr{t}_\lambda^{-1}\fr{d}$.
\end{proof}

Now, we return to the Euler product of $L(s)=L(s,A)$. Suppose that $A(\clo_F)=1$ and that $L(s)$ has a $\fr{p}$-partial Euler product as the following
\begin{align*}
    L(s)\eq L^\fr{p}(s)\;L_\fr{p}(s)
\end{align*}
where
\begin{align*}
    L_\fr{p}(s)\eq\sum_{j=0}^{\infty}A(\fr{p}^j)\cdot\N(\fr{p}^j)^{-s}\hspace{0.5cm}\text{and}\hspace{0.5cm}
    L^\fr{p}(s)\eq\sum_{\fr{p}\nmid\fr{a}}A(\fr{a})\cdot\N(\fr{a})^{-s}.
\end{align*}
Let $\cl{C}\in\Cl^+(F)$ be an equivalence class. Since $[\fr{p}]=[\clo_F]$ in $\Cl^+(F)$, the partial $L$-series $L(s,\cl{C})$ also has a $\fr{p}$-partial Euler product
\begin{align*}
    L(s,\cl{C})\eq L^\fr{p}(s,\cl{C})\;L_\fr{p}(s,\cl{C})
\end{align*}
such that 
\begin{align*}
    L_\fr{p}(s,\cl{C})\eq L_\fr{p}(s)\hspace{0.5cm}\text{and}\hspace{0.5cm}
    L^\fr{p}(s,\cl{C})\eq\sum_{\substack{\fr{a}\in\cl{C}\\\fr{p}\nmid\fr{a}}}A(\fr{a})\cdot\N(\fr{a})^{-s}.
\end{align*}

\begin{prop}
\label{localfactoreigenfunction}
Fix $\lambda$ with $1\leq\lambda\leq h$ and let $(\fr{p},p,\cl{F})$ be as before. Then, the following are equivalent
\begin{enumerate}[label=\rm{(\roman*)}]
    \item $T_{p,\cl{F}}^\infty f_\lambda=A(\fr{p})f_\lambda$;
    \item $L_\fr{p}(s,\cl{C}_\lambda)=(1-A(\fr{p})\N(\fr{p})^{-s}+\N(\fr{p})^{k_0-1-2s})^{-1}$ where $\cl{C}_\lambda=[\fr{t}_\lambda^{-1}\fr{d}]$.
\end{enumerate}
\end{prop}
\begin{proof}
It is a standard computation that (ii) is equivalent to the condition (ii) in \Cref{eigenfourier} with $C=A(\fr{p})$. The result then follows from \Cref{eigenfourier}.
\end{proof}

\begin{corol}
\label{eulercorol}
If the $L$-series $L(s)$ satisfies the Euler product as given in \Cref{mainthm}, then for each $1\leq\lambda\leq h$, $T_{p,\cl{F}}^\infty f_\lambda=A(\fr{p})f_\lambda$ for all $(\fr{p},p,\cl{F})$ as before.
\end{corol}

\section{Invariance of $f_\lambda$ under $\Gamma_0^+(\fr{t}_\lambda,\fr{n})$}
\label{gamma0-inv-sect}

In this section, we will prove the invariance of $f_\lambda$ under $\Gamma_0^+(\fr{t}_\lambda,\fr{n})$ under all the assumptions in \Cref{mainthm}. Throughout this section, we will always work under the assumptions on $L(s)$ in \Cref{mainthm}. The proof mainly uses the idea from \cite[Theorem~2.2]{conjhecke}. The original idea dealt with elliptic modular forms and we modify it to apply to Hilbert modular forms. As mentioned in the beginning, there is a small but easily fixable gap in the proof of \cite[Theorem~2.2]{conjhecke}, and we will provide our own fix at the end of this section. For the coherence of the discussion, we postpone one technical but important lemma to the end of this section, which is also closely related to our fix to the gap.

\begin{prop}
\label{gamma0-inv}
Let $\{\fr{t}_1,\ldots,\fr{t}_h\}$ be a set of representatives of $\Cl^+(F)$ such that each $\fr{t}_\lambda$ is a \textbf{prime} ideal. If the $L$-series $L(s)$ satisfies the assumptions on the analytic properties, the functional equations, and the Euler product in \Cref{mainthm}, then for each \mbox{$1\leq\lambda\leq h$}, the associated function $f_\lambda$ satisfies that $f_\lambda|_\mbf{k}\gamma=f_\lambda$ for all $\gamma\in\Gamma_0^+(\fr{t}_\lambda,\fr{n})$.
\end{prop}

Before proceeding to the proof, we introduce one more notational convention that will be used in the proof. We will extend the slash operator to the group ring $\comp[\GL_2^+(F)]$ by setting
\begin{align*}
    g|_\mbf{k}\bigg(\sum_j a_j[\gamma_j]\bigg)\deq\sum_j a_j\thin g|_\mbf{k}\gamma_j
\end{align*}
for $a_j\in\comp$ and $\gamma_j\in\GL_2^+(F)$. Here we write $\gamma$ as an element in $\GL_2^+(F)$ and $[\gamma]$ as an element in $\comp[\GL_2^+(F)]$ to avoid confusion.
\\

Now, we go back to the proof. Recall that $\SG_0(\fr{t}_\lambda\fr{n})=\{\sm{a&b\\c&d}\in\SL_2(\clo_F)\mid c\in\fr{t}_\lambda\fr{n}\}$. We will first show that each $f_\lambda$ is invariant under $\SG_0(\fr{t}_\lambda\fr{n})$, for which there is no need to assume that each $\fr{t}_\lambda$ is a prime ideal.

Note that under the two assumptions in \Cref{mainthm}, we have obtained that for all $1\leq\lambda\leq h$,
\begin{enumerate}[label=\rm{(\alph*)}]
    \item $f_\lambda|_\mbf{k}T^\alpha=f_\lambda$ for all $\alpha\in\fr{t}_\lambda^{-1}$ (\Cref{basic});
    \item $f_\lambda|_\mbf{k}E_\eps=f_\lambda$ for all $\eps\in\clo_+^\times$ (\Cref{basic});
    \item $f_\lambda|_\mbf{k}W_{q_\lambda}=\epsilon f_{\wt{\lambda}}$ for any totally positive generator $q_\lambda$ of $\fr{t}_\lambda\fr{t}_{\wt{\lambda}}\fr{n}$ (\Cref{fqncorol});
    \item $f_\lambda|_\mbf{k}A_\beta=f_\lambda$ for all $\beta\in\fr{t}_\lambda\fr{n}$ (\Cref{fqncorol});
    \item $T_{p,\cl{F}}^\infty f_\lambda=A(\fr{p})f_\lambda$ for all $(\fr{p},p,\cl{F})$ as before (\Cref{eulercorol}).
\end{enumerate}
As $F$ is a totally real field of degree $r>1$, it follows from (a), (d), and \Cref{sgamma1-generator} that each $f_\lambda$ is already invariant under $\SG_1(\fr{t}_\lambda\fr{n})$, which is a finite index subgroup of $\SG_0(\fr{t}_\lambda\fr{n})$.

For the next lemma, we do a renormalization of the Hecke operators at infinity. Let $(\fr{p},p,\cl{F})$ be as before. Define
\begin{align*}
    \cl{T}_{p,\cl{F}}\deq\bm{p&0\\0&1}+\sum_{\alpha\in\cl{F}}\bm{1&\alpha\\0&p}.
\end{align*}
Indeed, $f_\lambda|_\mbf{k}\cl{T}_{p,\cl{F}}=\frac{A(\fr{p})}{\N(\fr{p})^{k_0/2-1}}f_\lambda$ by (e).

\begin{lemma}
\label{primesum}
Let $(\fr{p},p,\cl{F})$ be as before with $\fr{p}\nmid(q_\lambda)$. Let $\cl{F}^\times=\cl{F}\bslash\{0\}$. Then,
\begin{align*}
    \sum_{\alpha\in \cl{F}^\times}f_\lambda|_\mbf{k}([\gamma_{p,-\alpha}]-1)[T^{\alpha/p}]\eq 0,
\end{align*}
where $\gamma_{p,-\alpha}=\sm{p&-\alpha\\-\beta q_\lambda&\frac{\alpha\beta q_\lambda+1}{p}}$ and $\beta$ is the unique element in $\cl{F}^\times$ such that $p|\alpha\beta q_\lambda+1$.
\end{lemma}
\begin{proof}
It is easy to check that $f_\lambda|_\mbf{k}[W_{q_\lambda}]\cl{T}_{p,\cl{F}}[W_{q_\lambda}^{-1}]=f_\lambda|_\mbf{k}\cl{T}_{p,\cl{F}}$. Hence,
\begin{align*}
    0&\eq f_\lambda|_\mbf{k}[W_{q_\lambda}]\cl{T}_{p,\cl{F}}[W_{q_\lambda}^{-1}]-f_\lambda|_\mbf{k}\cl{T}_{p,\cl{F}} \\
    &\eq f_\lambda|_\mbf{k}\sum_{\beta\in F^\times}\bm{p&0\\-\beta q_\lambda&1}-f|_\mbf{k}\sum_{\alpha\in\cl{F}^\times}\bm{1&\alpha\\0&p} \\
    &\eq f_\lambda|_\mbf{k}\sum_{\alpha\in\cl{F}^\times}\bigg(\bm{p&-\alpha\\-\beta q_\lambda&\frac{\alpha\beta q_\lambda+1}{p}}-1\bigg)[T^{\alpha/p}]\bm{1&0\\0&p}.
\end{align*}
The result then follows.
\end{proof}

We will now begin to extend the group that $f_\lambda$ is invariant under from $\SG_1(\fr{t}_\lambda\fr{n})$ to $\SG_0(\fr{t}_\lambda\fr{n})$. This method does not rely on the exact shape of $\SG_1(\fr{t}_\lambda\fr{n})$ but rather that it is a finite index subgroup containing $T^\alpha$ and $A_\beta$ for all $\alpha\in\clo_F$ and $\beta\in\fr{t}_\lambda\fr{n}$. We will thus simply write $H=\SG_1(\fr{t}_\lambda\fr{n})$ throughout the proof. However, it may be possible to simplify the proof using ad hoc properties of $\SG_1(\fr{t}_\lambda\fr{n})$.

Let $u=[\SG_0(\fr{t}_\lambda\fr{n}):H]$ and choose $\{\gamma_1,\ldots,\gamma_u\}$ such that $\SG_0(\fr{t}_\lambda\fr{n})=\bigsqcup_{i=1}^u H\gamma_i$. We may replace $\gamma_i$ by~$A_\beta\gamma_i$ or $T^\alpha\gamma_i$ so that $\gamma_i=\sm{a_i&b_i\\c_i&d_i}$ with $a_i,c_i\neq0$. Let $H_i=\gamma_i^{-1}H\gamma_i\cap H$ and
\begin{align*}
    g_i\eq f_\lambda|_\mbf{k}\gamma_i-f_\lambda
\end{align*}
for $1\leq i\leq u$. Then, $g_i$ is invariant under $H_i$ for $1\leq i\leq u$. 

\begin{lemma}
\label{gi-fourier-expansion}
There exists $m\in\inte^+$ such that $T^{\alpha m}\in H_i$ for all $\alpha\in\clo_F$ and all $1\leq i\leq u$.
\end{lemma}
\begin{proof}
Let $\{\delta_1,\ldots,\delta_r\}$ be a basis of $\clo_F$ over~$\inte$. For each $1\leq i\leq u$ and $1\leq l\leq r$, since \mbox{$[\SG_0(\fr{t}_\lambda\fr{n}):\gamma_i^{-1}H\gamma_i]=[\SG_0(\fr{t}_\lambda\fr{n}):H]<\infty$}, there exists $m_{il}\in\pos$ such that $\gamma_i^{-1}H\gamma_i(T^{\delta_l})^{m_{il}}=\gamma_i^{-1}H\gamma_i$. Set $m=\lcm_{1\leq i\leq u,1\leq l\leq r}\{m_{il}\}$. Then, \mbox{$\gamma_i^{-1}H\gamma_iT^{\alpha m}=\gamma_i^{-1}H\gamma_i$} for all $\alpha\in\clo_F$ and all $1\leq i\leq u$. The result then follows.
\end{proof}

By \Cref{gi-fourier-expansion}, each $g_i$ has a Fourier series expansion
\begin{align*}
    \big(f_\lambda|_\mbf{k}([\gamma_i]-1)\big)(z)\eq g_i(z)\eq\sum_{\xi\in\fr{d}^{-1}}\mu_i(\xi)\thin e\big(\tr\big(\tfrac{\xi z}{m}\big)\big)
\end{align*}
for some $\mu_i(\xi)\in\comp$.

Let $(\fr{p},p,\cl{F})$ be as before with $\fr{p}\nmid(q_\lambda)$. Write $\cl{F}^\times=\cl{F}\bslash\{0\}$. Then, by \Cref{primesum},
\begin{align*}
    0&\eq\sum_{\alpha\in\cl{F}^\times}f_\lambda|_\mbf{k}([\gamma_{p,-\alpha}]-1)[T^{\alpha/p}] \\
    &\eq\sum_{i=1}^u\sum_{\substack{\alpha\in\cl{F}^\times \\ \gamma_{p,-\alpha}\in H\gamma_i}}f_\lambda|_\mbf{k}([\gamma_i]-1)[T^{\alpha/p}] \\
    &\eq\sum_{i=1}^u\sum_{\substack{\alpha\in\cl{F}^\times \\ \gamma_{p,-\alpha}\in H\gamma_i}}\sum_{\xi\in\fr{d}^{-1}}\mu_i(\xi)\; e\big(\tr\big(\tfrac{\xi(z+\alpha/p)}{m}\big)\big) \\
    &\eq\sum_{\xi\in\fr{d}^{-1}}\bigg(\sum_{i=1}^u\mu_i(\xi)\sum_{\substack{\alpha\in\cl{F}^\times \\ \gamma_{p,-\alpha}\in H\gamma_i}}e\big(\tr\big(\tfrac{\xi\alpha}{pm}\big)\big)\bigg)\thin e\big(\tr\big(\tfrac{\xi z}{m}\big)\big).
\end{align*}
This implies that for all $\xi\in\fr{d}^{-1}$,
\begin{align}
\label{lambdaexposum}
    \sum_{i=1}^u\mu_i(\xi)\bigg(\sum_{\substack{\alpha\in\cl{F}^\times \\ \gamma_{p,-\alpha}\in H\gamma_i}}e\big(\tr\big(\tfrac{\xi\alpha}{pm}\big)\big)\bigg)=0.
\end{align}

We will now prove that $\mu_i(\xi)=0$ for all $\xi\in\fr{d}^{-1}\bslash\{0\}$ so that $g_i$ is a constant for each $1\leq i\leq u$. Fix an element $\xi\in\fr{d}^{-1}\bslash\{0\}$. For each $1\leq j\leq u$, write $\gamma_j=\sm{a_j&b_j\\c_j&d_j}$ as before. Note that $a_j\neq 0$, $c_j\neq 0$, and $((a_j),(c_j))=1$ by our choice of $\gamma_j$. By \Cref{dirichlet}, we may choose $r_j\in\clo_F$ for each $1\leq j\leq u$ such that
\begin{enumerate}[label=(P\arabic*)]
    \item each $p_j=a_j+r_jc_j$ is totally positive and $(p_j)$ is a prime of degree $1$;
    \item $(p_j)\nmid mq_\lambda\N(\fr{d})\fr{n}$ and $(p_j)\nmid\xi\fr{d}$ for all $1\leq j\leq u$;
    \item all $\N(p_1),\ldots,\N(p_u)$ are distinct rational primes.
\end{enumerate}

Let $\alpha_j=-b_j-r_jd_j$. Note that $((p_j),(\alpha_j))=1$ and $(p_j)$ is a prime of degree~$1$, so by \Cref{standardfp}, $\cl{F}_j:=\{a\alpha_j\mid a\in\inte,1\leq a\leq\N(p_j)-1\}$ is a set of representatives of $(\clo_F/(p_j))^\times$. For $\alpha\in\cl{F}_j$, we write
\begin{align*}
    \gamma_{p_j,-\alpha}\eq\m{p_j&-\alpha\\-\beta q_\lambda&\frac{\alpha\beta q_\lambda+1}{p_j}}
\end{align*}
as before, where $\beta\in\cl{F}_j$ is the unique element such that $p_j|\alpha\beta q_\lambda+1$. Define
\begin{align*}
    \cl{F}_{i,j}\eq \{\alpha\in \cl{F}_j\mid \gamma_{p_j,-\alpha}\in H\gamma_i\}.
\end{align*}
Note that $T^{r_j}\gamma_j=\sm{p_j&-\alpha_j\\**&*}$ so there exists some $t_j\in\fr{t}_\lambda\fr{n}$ such that $\gamma_{p_j,-\alpha_j}=\sm{1&0\\t_j&1}T^{r_j}\gamma_j$.
Then, in particular, $\alpha_j\in \cl{F}_{j,j}$ as $\gamma_{p_j,-\alpha_j}=A_{t_j}T^{r_j}\gamma_j\in H\gamma_j$, so $\cl{F}_{j,j}\neq\emptyset$. Applying $p=p_j$ and $\cl{F}^\times=\cl{F}_j$ to \Cref{lambdaexposum}, we obtain that
\begin{align}
\label{lambdaexposumij}
    \sum_{i=1}^u\mu_i(\xi)\bigg(\sum_{\alpha\in \cl{F}_{i,j}}e\big(\tr\big(\tfrac{\xi\alpha}{p_jm}\big)\big)\bigg)\eq 0.
\end{align}
We will rewrite this equation in order to apply \Cref{nonzerodet} below.

Let $q_j=\N(p_j)$ for each $1\leq j\leq u$. Define
\begin{align*}
    s_{i,j}\deq\{a\in\inte\mid 1\leq a\leq q_j-1, \;a\alpha_j\in \cl{F}_{i,j}\}\;\subseteq\;\{1,\ldots,q_j-1\}.
\end{align*}
Write $n_j=\tr\big(\xi\alpha_j\cdot\tfrac{q_j}{p_j}\big)\in\inte$. Note that $q_j\nmid m$ by (P2), and $q_j\nmid n_j$ by \Cref{npnmid} since $(p_j)\nmid \alpha_j\xi\fr{d}$ and $(p_j)\nmid(\N(\fr{d}))$. Then, for $a\in\{1,\ldots,q_j-1\}$, 
\begin{align*}
    \tr\bigg(\frac{\xi a\alpha_j}{p_jm}\bigg)\eq\frac{a}{mq_j}\cdot\tr\bigg(\xi\alpha_j \cdot\frac{q_j}{p_j}\bigg)\eq\frac{n_ja}{mq_j}
\end{align*}
and hence
\begin{align*}
    S_{i,j}\deq\sum_{\alpha\in \cl{F}_{i,j}}e\big(\tr\big(\tfrac{\xi\alpha}{p_jm}\big)\big)\eq\sum_{a\in s_{i,j}}e\big(\tr\big(\tfrac{\xi a\alpha_j}{p_jm}\big)\big)\eq\sum_{a\in s_{i,j}}e\big(\tfrac{n_ja}{mq_j}\big).
\end{align*}

To summarize, we obtain the data $(q_j,m_j=m,n_j,s_{i,j})_{i,j=1}^u$ which satisfies that
\begin{enumerate}[label=\rm{(\arabic*)}]
    \item all $q_j$ are distinct rational primes by (P1) and (P3);
    \item $q_j\nmid m_jn_j$ for all $j$ by (P1), (P2), and \Cref{npnmid};
    \item $q_j\nmid m_i$ for all $i$ and $j$ by (P2);
    \item $s_{i,j}\subseteq\{1,\ldots,q_j-1\}$ by the definition of $s_{i,j}$;
    \item $s_{i,j}\cap s_{i^\prime,j}=\emptyset$ for $i\neq i^\prime$ by the definition of $s_{i,j}$;
    \item $s_{i,i}\neq\emptyset$ for all $i$ as $\cl{F}_{i,i}\neq\emptyset$ for all $i$.
\end{enumerate}
Therefore, the data $(q_j,m_j=m,n_j,s_{i,j})_{i,j=1}^u$ satisfies the requirements in \Cref{nonzerodet}, so \mbox{$\det((S_{i,j})_{i,j=1}^u)\neq 0$}. By \Cref{lambdaexposumij},
\begin{align*}
    \sum_{i=1}^u\mu_i(\xi)\thin S_{i,j}\eq 0.
\end{align*}
for all $1\leq j\leq u$. This then implies that $\mu_i(\xi)=0$ for all $1\leq i\leq u$.

Note that $\xi\in\fr{d}^{-1}\bslash\{0\}$ is fixed but arbitrary. Hence, we obtain that $\mu_i(\xi)=0$ for all $\xi\in\fr{d}^{-1}\bslash\{0\}$ and all $1\leq i\leq u$. Thus, all $g_i$ are constant. Now, each $g_i$ is invariant under $H_i=\gamma_i^{-1}H\gamma_i\cap H$, which is a finite index subgroup of $\SL_2(\clo_F)$ since $H$ is. That is, each $g_i$ is a classical Hilbert modular form of weight $\mbf{k}$ on $H_i$. Since $\mbf{k}\in(2\inte^+)^r$ and $g_i$ is constant, $g_i$ must be zero. Hence, $g_i=0$ for all $1\leq i\leq u$, meaning that $f_\lambda|_\mbf{k}\gamma_i=f_\lambda$ for all $1\leq i\leq u$. Thus, $f_\lambda$ is invariant under $\SG_0(\fr{t}_\lambda\fr{n})$. 

To further extend the group to $\Gamma_0^+(\fr{t}_\lambda,\fr{n})$, we choose $\{\fr{t}_1,\ldots,\fr{t}_h\}$ such that each $\fr{t}_\lambda$ is a prime ideal (such representatives exist by Chebotarev's density theorem). Then, by \Cref{gamma0-generator}, $\Gamma_0^+(\fr{t}_\lambda,\fr{n})$ is generated by $\SG_0(\fr{t}_\lambda\fr{n})$, $T^\alpha$, and $E_\eps$ for all $\alpha\in\fr{t}_\lambda^{-1}$ and $\eps\in\clo_+^\times$. In particular, $f_\lambda$ is invariant under $\Gamma_0^+(\fr{t}_\lambda,\fr{n})$, so \Cref{gamma0-inv} follows.

In conclusion, under all the assumptions in \Cref{mainthm}, $\mbf{f}=(f_1,\ldots,f_h)$ is an adelic Hilbert modular form of weight $\mbf{k}$ and level $K_0(\fr{n})$. It then follows that $L(s)=L(s,\mbf{f})$ is the $L$-function of a Hilbert modular form of weight $\mbf{k}$ and level $K_0(\fr{n})$. It is also easy to see that throughout the proof of \Cref{gamma0-inv}, we only use the local Euler factors at primes $\fr{p}$ that are of degree~$1$, with $[\fr{p}]=[\clo_F]$ in $\Cl^+(F)$, and $\fr{p}\nmid\fr{n}\fr{d}^2$. Hence, the full Euler product in \Cref{mainthm} can be relaxed to a partial Euler product at those primes.
\\

Now, we introduce the technical lemma for showing that for fixed $\xi\in\fr{d}^{-1}\bslash\{0\}$, $\mu_i(\xi)=0$ for all $1\leq i\leq u$. This lemma is a slight generalization of \cite[Lemma~4.5]{conjhecke} and the proof also follows from the same idea.

\begin{lemma}
\label{nonzerodet}
Let $u\in\pos$. For each $1\leq j\leq u$, let $q_j$ be a rational prime, $n_j$ be an integer, and $m_j$ be a positive integer such that
\begin{enumerate}[label=\rm{(A\arabic*)}]
    \item all $q_j$ are distinct;
    \item $q_j\nmid m_jn_j$ for all $j$;
    \item $q_j\nmid m_i$ for all $i$ and $j$.
\end{enumerate}
Let $s_{i,j}\subseteq\{1,\ldots,q_j-1\}$ for each $1\leq i\leq u$ such that
\begin{enumerate}[label=\rm{(B\arabic*)}]
    \item $s_{i,j}\cap s_{i^\prime,j}=\emptyset$ for $i\neq i^\prime$;
    \item $s_{i,i}\neq\emptyset$ for all $i$.
\end{enumerate}
Write $S_{i,j}=\sum_{a\in s_{i,j}}e\big(\tfrac{n_ja}{m_jq_j}\big)$ for each $i,j$. Then, $\det((S_{i,j})_{i,j=1}^u)\neq 0$.
\end{lemma}
\begin{proof}
Note that replacing $m_j$ and $n_j$ with $\frac{m_j}{(m_j,n_j)}$ and $\frac{n_j}{(m_j,n_j)}$ does not affect the result, so we may assume that $(m_j,n_j)=1$ for all $1\leq j\leq u$. We also write $\zeta_j=e\big(\tfrac{n_j}{m_jq_j}\big)$ so that $S_{i,j}=\sum_{a\in s_{i,j}}\zeta_j^a$. Note that each $\zeta_j$ is a primitive $(m_jq_j)$-th root of unity by the coprime assumptions. We will now prove the result by induction on $u$.

Suppose that $u=1$. It suffices to show that $S_{1,1}\neq 0$. Write $Q(X)=\sum_{a\in s_{1,1}}X^{a-1}$. Then, $S_{1,1}=\zeta_1\cdot Q(\zeta_1)$, and $Q$~is a nonzero polynomial in $\rat[X]$ of degree $\leq q_1-2$ as \mbox{$\emptyset\neq s_{1,1}\subseteq\{1,\ldots,q_1-1\}$}. Note that
\begin{align*}
    [\rat(\zeta_1):\rat]\eq\vphi(m_1q_1)\eq\vphi(m_1)\vphi(q_1)\;\geq\;\vphi(q_1)\eq q_1-1\;>\;q_1-2.
\end{align*}
Since $\deg Q\leq q_1-2$ and $Q$ is nonzero, this implies that $Q(\zeta_1)\neq 0$, and hence $S_{1,1}\neq 0$.

Now, suppose that $u\geq 2$. Note that $S_{i,1}=\sum_{a\in s_{i,1}}\zeta_1^a$ so $\det((S_{i,j})_{i,j=1}^u)$ can be expressed as $P(\zeta_1)$ for some $P\in\rat(\zeta_2,\ldots,\zeta_u)[X]$ with $\deg P\leq q_1-1$. By (B1) and (B2), there exists an element $a$ such that $a\in s_{1,1}$ and $a\notin s_{i,1}$ for all $i\neq 1$. The coefficient of the term $X^a$ in $P$ is then the determinant of the matrix $(S_{i,j})_{i,j=2}^u$, which is nonzero by the induction hypothesis. Hence, $P$ is nonzero, and we may write $P(X)=X\cdot Q(X)$ for some nonzero $Q\in\rat(\zeta_2,\ldots,\zeta_u)[X]$ with $\deg Q\leq q_1-2$ as $s_{i,1}\subseteq\{1,\ldots,q_1-1\}$.

Write $m=\lcm(m_1,\ldots,m_u)$ and \mbox{$m^\prime=\lcm(m_2,\ldots,m_u)$}. By the coprime assumptions, 
\begin{align*}
    [\rat(\zeta_1,\ldots,\zeta_u):\rat(\zeta_2,\ldots,\zeta_u)]&\eq\frac{\vphi(mq_1\cdots q_u)}{\vphi(m^\prime q_2\cdots q_u)} \\
    &\eq \frac{\vphi(m)}{\vphi(m^\prime)}\cdot\vphi(q_1)\\
    &\;\geq\;\vphi(q_1)\eq q_1-1\;>\;q_1-2
\end{align*}
where $\vphi(m)/\vphi(m^\prime)\geq 1$ since $m^\prime|m$. Since $\deg Q\leq q_1-2$ and $Q$ is nonzero, this implies that $Q(\zeta_1)\neq 0$, and hence $\det((S_{i,j})_{i,j=1}^u)=P(\zeta_1)=\zeta_1\cdot Q(\zeta_1)\neq 0$.
\end{proof}

To finish this section, we make a short discussion on the gap in the proof of \cite[Theorem~2.2]{conjhecke} and provide our own fix. The proof uses Dirichlet's theorem to obtain primes $q_i$ and integers $a_i$ such that $\gamma_{q_i,a_i}=\sm{q_i&a_i\\**&*}\in H g_i$ and then plugs these $\gamma_{q_i,a_i}$ into Equation (2.3) in their paper. However, these integers $a_i$ \textit{a priori} do not lie in $\{1,2,\ldots,q_i-1\}$ while Equation (2.3), as a corollary of Equation (2.1) in their paper, requires that $a_{il}\in\{1,2,\ldots,q_i-1\}$. It is thus not possible to simply plug in $\gamma_{q_i,a_i}$ into Equation (2.3).

The fix that we offer here is essentially what we do in the case of Hilbert modular forms. We first prove that Equation (2.1) still holds if we replace the summation set by any set of representatives $\cl{F}^\times$ of $(\inte/q\inte)^\times$. In their paper, this indeed follows from \cite[Lemma 4.1]{conjhecke}. Then, we choose primes $q_i$ and integers $a_i$ such that $\gamma_{q_i,a_i}\in H g_i$ as before, and for each $q_i$ we choose the set of representatives $\cl{F}_i=\{aa_i\mid 1\leq a\leq q_i-1\}$ of $(\inte/q_i\inte)^\times$. Now, we may apply Equation (2.1) with $\cl{F}^\times=\cl{F}_i$ and proceed as what we have done in this section. That is, if we write
\begin{align*}
    s_{i,j}\eq\{a\in\inte\mid 1\leq a\leq q_j-1,\; \gamma_{q_j,aa_j}\in H g_j\},
\end{align*}
then the same argument yields a new version of Equation (2.4) in their proof as the following
\begin{align*}
    \sum_{i}\bigg(\sum_{a\in s_{i,j}}e\bigg(\frac{na_ja}{mq_j}\bigg)\bigg)\lambda_i(n)\eq 0 \hspace{0.3cm} \text{for all }j.
\end{align*}
To obtain a similar result on the nonvanishing of the determinant, we prove a modified version of \cite[Lemma~4.5]{conjhecke}, i.e., \Cref{nonzerodet} in this section. We then apply \Cref{nonzerodet} to the data $(q_j,m_j=m,n_j=na_j,s_{i,j})_{i,j}$ to deduce the nonvanishing of the determinant, and hence obtain that \mbox{$\lambda_i(n)=0$} for all $i$.

\appendix

\section{Lemmas on prime ideals}
\label{degree1}

In this appendix, we will prove all the necessary lemmas on prime ideals. As these lemmas are general enough, we will simply work with general number fields, and as before, we let $K$ denote a general number field to distinguish it from $F$. 

We start by first recalling some definitions. For a set $S$ of primes of a number field $K$, define the \textit{natural density} of $S$ as
\begin{align*}
    \lim_{X\rightarrow\infty}\frac{\hash\{\fr{p}\text{ prime of }K\mid\N(\fr{p})\leq X,\;\fr{p}\in S\}}{\hash\{\fr{p}\text{ prime of }K\mid\N(\fr{p})\leq X\}}
\end{align*}
if it exists. A prime $\fr{p}$ of $K$ is called \textit{of degree $1$} if it lies over a rational prime $q$ which splits completely in $K$.

\begin{lemma}
\label{deg1}
The set of primes of degree $1$ in a number field $K$ has density $1$.
\end{lemma}
\begin{proof}
See e.g. \cite[Lemma~1]{ziegler}.
\end{proof}

\begin{lemma}
\label{local-quotient}
Let $K$ be a number field and $\fr{p}$ be a prime of $K$. Let $S\subseteq\clo_K-\fr{p}$ be a multiplicatively closed set. Then, the natural map $\clo_K/\fr{p}\rightarrow S^{-1}\clo_K/S^{-1}\fr{p}$ is an isomorphism.
\end{lemma}
\begin{proof}
Note that localization commutes with quotient (see \cite[\href{https://stacks.math.columbia.edu/tag/00CT}{Proposition 10.9.14 (00CT)}]{stacks-project}). That is, if we let $\ov{S}$ denote the image of $S$ under the quotient map $\clo_K\twoheadrightarrow\clo_K/\fr{p}$, then $\ov{S}^{-1}(\clo_K/\fr{p})\simarrow S^{-1}\clo_K/S^{-1}\fr{p}$. Note that $\ov{S}\subseteq(\clo_K/\fr{p})^\times$ since $S\subseteq\clo_K-\fr{p}$, so $\ov{S}^{-1}(\clo_K/\fr{p})=\clo_K/\fr{p}$.
\end{proof}

\begin{corol}
\label{ideal-local-quotient}
Let $K$ be a number field, $\fr{t}$ be a nonzero ideal of $K$, and $\fr{p}$ be a nonzero prime ideal of $K$ such that $\fr{p}\nmid(\N(\fr{t}))$. Then, the natural map $\clo_K/\fr{p}\rightarrow\fr{t}^{-1}/\fr{t}^{-1}\fr{p}$ is bijective.
\end{corol}
\begin{proof}
Write $T=\N(\fr{t})$ and let $S=\{T^n|n\in\inte_{\geq0}\}$. As $\fr{p}\nmid (T)$, $S$ is a multiplicatively closed subset of $\clo_K-\fr{p}$, so by \Cref{local-quotient}, $\clo_K/\fr{p}\simarrow S^{-1}\clo_K/S^{-1}\fr{p}$. Also, note that $T\in\fr{t}$, so $T\xi\in\clo_K$ for all $\xi\in\fr{t}^{-1}$, and hence $\fr{t}^{-1}\subseteq S^{-1}\clo_K$. In particular, we obtain a sequence of maps
\begin{align*}
    \clo_K/\fr{p}\rightarrow \fr{t}^{-1}/\fr{t}^{-1}\fr{p}\rightarrow S^{-1}\clo_K/S^{-1}\fr{p}. 
\end{align*}
Since $\clo_K/\fr{p}\simarrow S^{-1}\clo_K/S^{-1}\fr{p}$, this forces the map $\clo_K/\fr{p}\rightarrow\fr{t}^{-1}/\fr{t}^{-1}\fr{p}$ to be bijective.
\end{proof}

Recall that an element $\alpha\in K$ is called \textit{totally positive} if $\sigma(\alpha)>0$ for all embeddings $\sigma:K\hookrightarrow\real$. In particular, if $\alpha\in K$ is totally positive, then $\N(\alpha)>0$. In the following three lemmas, $p\in\clo_K$ is a fixed totally positive element such that $(p)$ is a prime of degree $1$.

\begin{lemma}
\label{standardfp}
Let $\alpha\in\clo_K$ with $p\nmid\alpha$. Then, the set of representatives of~$\clo_K/(p)$ can be chosen as $\cl{F}_p^\alpha=\{a\alpha\mid a\in\inte, 0\leq a\leq \N(p)-1\}$.
\end{lemma}
\begin{proof}
Write $q=\N(p)$ a rational prime. For any $a,b\in \cl{F}_p^1$, $p|(a\alpha-b\alpha)$ implies $p|(a-b)$ and hence $\N(p)|\N(a-b)$, so $q|(a-b)^r$. As both $q$ and $a-b$ are integers and $q$ is a rational prime, this implies that $q|(a-b)$, and hence $a=b$ by the choice. It follows that $\cl{F}_p^\alpha$ is a set of representatives of $\clo_K/(p)$ as $|\cl{F}_p^\alpha|=\N(p)$.
\end{proof}

\begin{lemma}
\label{xi-represent}
Let $\fr{t}$ be a nonzero ideal of $\clo_K$ such that $p\nmid\N(\fr{t})$. Then, every $\xi\in\fr{t}^{-1}$ can be written as $\xi=n+p\eta$ for some $n\in\cl{F}_p^1=\{0,1,\ldots,\N(p)-1\}$ and $\eta\in\fr{t}^{-1}$.
\end{lemma}
\begin{proof}
Since $p\nmid\N(\fr{t})$, by \Cref{ideal-local-quotient}, the natural map $\clo_K/(p)\rightarrow\fr{t}^{-1}/\fr{t}^{-1}(p)$ is bijective. By \Cref{standardfp}, $\cl{F}_p^1$ is a set of representatives of $\clo_K/(p)$, so it is also a set of representatives of $\fr{t}^{-1}/\fr{t}^{-1}(p)$. That is, for any $\xi\in\fr{t}^{-1}$, there exist $n\in\cl{F}_p^1$ and $\eta\in\fr{t}^{-1}$ such that $\xi=n+p\eta$, which proves the result.
\end{proof}

\begin{lemma}
\label{npnmid}
Suppose that $p\nmid\N(\fr{d})$. Then, for any $\xi\in\fr{d}^{-1}$, $(p)|\thin\xi\fr{d}$ if and only if $\N(p)\thin|\thin\tr\big(\frac{\xi\N(p)}{p}\big)$.
\end{lemma}
\begin{proof}
First, $\tr\big(\frac{\xi\N(p)}{p}\big)\in\inte$ as $p\thin|\N(p)$. If $(p)|\thin\xi\fr{d}$, then $\xi/p\in\fr{d}^{-1}$, so $\tr\big(\frac{\xi\N(p)}{p}\big)=\N(p)\cdot\tr(\xi/p)$ is divisible by $\N(p)$. Now, suppose that $\N(p)\thin|\thin\tr\big(\frac{\xi\N(p)}{p}\big)$ so that $\tr\big(\frac{\xi}{p}\big)\in\inte$ and suppose also that $(p)\nmid\xi\fr{d}$.  By \Cref{xi-represent}, there exist $n\in\cl{F}_p^1$ and $\eta\in\fr{d}^{-1}$ such that $\xi=n+p\eta$. Then,
\begin{align*}
    \tr\bigg(\frac{\xi}{p}\bigg)\eq\tr\bigg(\frac{n+p\eta}{p}\bigg)\eq n\cdot\tr\bigg(\frac{1}{p}\bigg)+\tr(\eta).
\end{align*}
Here $\tr\big(\frac{\xi}{p}\big)\in\inte$ by the assumption and $\tr(\eta)\in\inte$ since $\eta\in\fr{d}^{-1}$. Hence, $n\cdot\tr(1/p)\in\inte$, where $n\neq 0$ since $(p)\nmid\xi\fr{d}$.

We will now prove that $\tr(\alpha/p)\in\inte$ for all $\alpha\in\clo_K$. First, for any $m\in\inte$, since $(n,\N(p))=1$, there exist $a,b\in\inte$ such that $m=an+b\thin\N(p)$. Hence,
\begin{align*}
    \tr\bigg(\frac{m}{p}\bigg)\eq\tr\bigg(\frac{an+b\thin\N(p)}{p}\bigg)\eq a\cdot n\cdot\tr\bigg(\frac{1}{p}\bigg)+b\cdot\tr\bigg(\frac{\N(p)}{p}\bigg)\in\inte.
\end{align*}
In particular, $\tr(m/p)\in\inte$ for all $m\in\cl{F}_p^1$. Now, as $\cl{F}_p^1$ is a set of representatives of $\clo_K/(p)$ by \Cref{standardfp}, each $\alpha\in\clo_K$ can be written as $\alpha=m+p\beta$ for some $m\in\cl{F}_p^1$ and $\beta\in\clo_K$. Hence,
\begin{align*}
    \tr\bigg(\frac{\alpha}{p}\bigg)\eq\tr\bigg(\frac{m+p\beta}{p}\bigg)\eq \tr\bigg(\frac{m}{p}\bigg)+\tr(\beta)\in\inte.
\end{align*}
Therefore, $\tr(\alpha/p)\in\inte$ for all $\alpha\in\clo_K$, so $1/p\in\fr{d}^{-1}$. In particular, $\frac{\N(\fr{d})}{p}\in\clo_K$ since $\N(\fr{d})\in\fr{d}$. This implies that $\N(p)|\N(\fr{d})^r$, and hence $p\thin|\N(\fr{d})^r$ since $p\thin|\N(p)$, which leads to contradiction as $p\nmid\N(\fr{d})$ by the assumption.
\end{proof}

The last two lemmas deal with Dirichlet's theorem on arithmetic progressions for general number fields. The next lemma is a fact on ray class groups that we will need for the proof of the last lemma.

\begin{lemma}
\label{rayclass}
Let $\fr{m}=\fr{m}_f\fr{m}_\infty$ be a modulus of $K$ where $\fr{m}_f$ is the non-archimedean part and $\fr{m}_\infty$ is the archimedean part. Let $C_\fr{m}$ be the ray class group for the modulus $\fr{m}$. Then, two nonzero integral ideals $\fr{a}$ and $\fr{b}$ correspond to the same class in $C_\fr{m}$ if and only if there exist nonzero $\alpha,\beta\in\clo_K$ such that
\begin{enumerate}[label=\rm{(\roman*)}]
    \item $\alpha\fr{a}=\beta\fr{b}$;
    \item $\alpha\equiv\beta\equiv1\pmod{\fr{m}_f}$;
    \item $\alpha$ and $\beta$ have the same sign for every real prime dividing $\fr{m}_\infty$.
\end{enumerate}
\end{lemma}
\begin{proof}
See Proposition 1.6 in \cite[Chapter~V]{milne}.
\end{proof}

\begin{lemma}
\label{dirichlet}
Let $a,c\in\clo_K$ be nonzero with $((a),(c))=1$. Then, there exist infinitely many totally positive elements $p\in\clo_K$ in the set \mbox{$\{a+nc\mid n\in\clo_K\}$} with distinct $(p)$ such that $(p)$ is a prime of degree $1$.
\end{lemma}
\begin{proof}
Note that $c|\N(c)$. Let $d\in\inte^+$ and $n=d\cdot\frac{|\N(c)|}{c}\in\clo_K$ so that $a+nc=a+d|\N(c)|$. By choosing $d$ large enough, we may replace $a$ by $a+d|\N(c)|$ to assume that $a$ is totally positive.

Let $\fr{m}=\fr{m}_f\fr{m}_\infty$ be a modulus of $K$ such that $\fr{m}_f=(c)$ and $\fr{m}_\infty$ is the product of all real places of $K$. Let $C_\fr{m}$ be the ray class group modulo~$\fr{m}$, let $L$ be the class field of $C_\fr{m}$, and let $\vphi:C_\fr{m}\rightarrow\Gal(L/K)$ be the Artin isomorphism. By Chebotarev's density theorem, the set of prime ideals $\fr{p}$ of $K$ satisfying that $\vphi([\fr{p}]_\fr{m})=\vphi([(a)]_\fr{m})$ (and hence $[\fr{p}]_\fr{m}=[(a)]_\fr{m}$) has positive density, where $[\fr{a}]_\fr{m}$ denotes the equivalence class of $\fr{a}$ in~$C_\fr{m}$. Since the set of prime ideals of degree $1$ in $K$ has density $1$ by \Cref{deg1}, this implies that there exist infinitely many prime ideals $\fr{p}$ of $K$ of degree $1$ such that $[\fr{p}]_\fr{m}=[(a)]_\fr{m}$. 

Let $\fr{p}$ be a prime ideal of $K$ of degree $1$ with $[\fr{p}]_\fr{m}=[(a)]_\fr{m}$. By \Cref{rayclass}, there exist $\alpha,\beta\in\clo_K$ such that (i)~$\alpha(a)=\beta\fr{p}$, (ii) $\alpha\equiv\beta\equiv1\pmod{c}$, and (iii) $\alpha/\beta$ is totally positive. Since $\fr{p}$ and $(a)$ lie in the same equivalence class in $C_\fr{m}$, they also lie in the same equivalence class in $\Cl^+(K)$ under the surjection $C_\fr{m}\twoheadrightarrow C_{\fr{m}_\infty}=\Cl^+(K)$, so $\fr{p}=(p)$ for some totally positive element $p\in\clo_K$. By~(i), there exist some $\eps\in\clo_K^\times$ such that $\alpha a=\eps\beta p$. Note that $\eps p=(\alpha/\beta)a$ is totally positive by (iii) and the assumption that $a$ is totally positive, so we may replace $p$ by $\eps p$ so that $\alpha a=\beta p$. By (ii), there exist $n_\alpha,n_\beta\in\clo_K$ such that $\alpha=1+n_\alpha c$ and $\beta=1+n_\beta c$. Rewriting $\alpha a=\beta p$ then gives that $a+(n_\alpha a-n_\beta p)c=p$. The result then follows.
\end{proof}

\bibliographystyle{amsalpha}
\bibliography{ref}

\end{document}